\documentclass[11pt, a4paper]{article}

\usepackage[utf8]{inputenc}
\usepackage[english]{babel}
\usepackage[T1]{fontenc}
\usepackage[colorlinks = true,
            linkcolor = blue,
            urlcolor  = blue,
            citecolor = blue,
            anchorcolor = blue,
            hyperfootnotes=false]{hyperref}
\usepackage[a4paper,width=150mm,top=25mm,bottom=25mm]{geometry}
\usepackage{fancyhdr}
\usepackage[symbol]{footmisc}

\usepackage{amsmath,amsthm,amssymb, enumitem, mathtools, mathrsfs}
\usepackage{siunitx}
\setitemize{noitemsep}

\usepackage{tikz}
\usepackage{pgfplots}
\pgfplotsset{compat=1.16} 
\usepackage{tikz-cd}
\usepackage{caption}
\usepackage[all]{xy}
\usepackage{graphicx}
\usepackage{chngcntr}
\usepackage{booktabs}
\usepackage[labelformat=simple]{subcaption}

\usepackage{pdfpages}

\usepackage{csquotes}
\usepackage[
backend=biber,
sorting=anyt
]{biblatex}
\graphicspath{ {images/} }

\theoremstyle{plain}
\newtheorem{theorem}{Theorem}[section]
\newtheorem*{theorem*}{Theorem}
\newtheorem*{satz*}{Satz}
\newtheorem{proposition}[theorem]{Proposition}
\newtheorem{lemma}[theorem]{Lemma}
\newtheorem{corollary}[theorem]{Corollary}

\theoremstyle{definition}
\newtheorem{definition}[theorem]{Definition}

\theoremstyle{remark}
\newtheorem{remark}[theorem]{Remark}

\numberwithin{equation}{section}

\DeclareMathOperator*{\esssup}{ess\,sup}

\newcommand{\ev}{\mathbb{E}}
\newcommand{\pr}{\mathbb{P}}

\newcommand{\R}{\mathbb{R}}
\renewcommand{\P}{\mathcal{P}}
\newcommand{\F}{\mathcal{F}}
\renewcommand{\L}{\mathcal{L}}

\renewcommand{\d}{\mathrm{d}}

\newcommand{\define}{\mathpunct{:}}
\newcommand{\bb}[1]{\mathbb{#1}}
\renewcommand{\bf}[1]{\mathbf{#1}}
\renewcommand{\cal}[1]{\mathcal{#1}}

\begin{document}

\noindent
\begin{center}
    \Large
    \textbf{A Moving Boundary Problem for Brownian Particles\\ with Singular Forward-Backward Interactions}

    \vspace{1em}

    \normalsize
    Philipp Jettkant\footnote[1]{Department of Mathematics, Imperial College London, United Kingdom, \href{mailto:p.jettkant@imperial.ac.uk}{p.jettkant@imperial.ac.uk}.} \& Andreas S\o jmark\footnote[2]{Department of Statistics, London School of Economics, United Kingdom, \href{mailto:a.sojmark@lse.ac.uk}{a.sojmark@lse.ac.uk}.}

\end{center}

\vspace{1em}

\begin{abstract}
We introduce a system of Brownian particles, each absorbed upon hitting an associated moving boundary. The boundaries are determined by the conditional probabilities of the particles being absorbed before some final time horizon, given the current knowledge of the system. While the particles evolve forward in time, the conditional probabilities are computed backwards in time, leading to a specification of the particle system as a system of singular forward-backward SDEs coupled through hitting times. Its analysis leads to a novel type of tiered moving boundary problem. Each level of this PDE corresponds to a different configuration of unabsorbed particles, with the boundary and the boundary condition of a given level being determined by the solution of the preceding one. We establish classical well-posedness for this moving boundary problem and use its solution to solve the original forward-backward system and prove its uniqueness.
\end{abstract}



\section{Introduction}

We study a system of $N$ Brownian particles that evolve together with $N$ moving boundaries for a given interval of time $[0,T]$. On their own, the particles perform Brownian motion independently of each other. However, each particle is absorbed as soon as it collides with its associated boundary. The possibility of this occurrence will be what determines the evolution of the boundaries of all particles, thereby coupling the system. Thus, the interesting part of the problem lies in the specification of the moving boundaries. At the final time $T$, the positions of the boundaries are characterised by exactly the set $\mathcal{J}$ of particles that have been absorbed. The precise specification depends on the connections between the particles, expressed by a weighted adjacency matrix $\mathbf{D}=(D_{ij})_{ij}$. Given this, the moving boundary of the $i$th particle equals $\sum_{j\in \mathcal{J}} D_{ij}$ at time $T$.

If we were to let the $i$th boundary advance by the amount $D_{ij}$ upon the absorption of the $j$th particle (if it occurs), consistent with the above terminal value, we would obtain a version of the particle systems studied by Hambly, Ledger \& S{\o}jmark \cite{HamblyLedgerSojmark2019} and Nadtochiy \& Shkolnikov \cite{NadtochiyShkolnikov2020} (see also the related problems in \cite{ChayesSwindle1996, dembo2019criticality}). In this work, we instead consider an element of anticipation so that the boundaries not only reflect the realised effect of past absorptions, but also the expected effect of potential future absorptions given the current configuration of the system. At any time $t \in [0, T)$, the boundary of the $i$th particle therefore equals the weighted sum of the conditional probabilities $\mathbb{P}(\tau_j \leq T| \mathcal{F}_t)$, weighted according to $\mathbf{D}$, where $\tau_j$ is the time at which particle $j$ is absorbed and $(\mathcal{F}_t)_{t\in[0,T]}$ is the filtration generated by the $N$ Brownian motions. In particular, rather than growing monotonically upon absorptions, the boundaries advance or recede dynamically depending on the changing probabilities of these events.

\begin{figure}[htbp]
	\centering
	\begin{tikzpicture}
		\begin{axis}[
			width=9cm,
			height=8cm,
			xmin=0, xmax=9.5,
			ymin=0, ymax=2.2,
			grid=major,
			grid style={gray!30},
			xtick={0,1,2,3,4,5,6,7,8,9},
			xticklabels={0,,,,{$\tau_2$},,,,,{$T$}},
			ytick={0,1,2},
			axis lines=left,
			tick label style={font=\small}
			]
			
			\definecolor{warmorange}{rgb}{1.0, 0.6, 0.2}
			
			\addplot[
			color=warmorange,
			mark=o,
			mark size=2pt,
			line width=1pt,
			dotted,
			mark options={solid}
			] coordinates {
				(0, 1.0800)
				(1, 0.8981)
				(2, 0.6788)
				(3, 0.6187)
				(4, 0.5663)
			};
			
			\addplot[
			color=warmorange,
			mark=o,
			mark size=2pt,
			line width=1pt,
			solid,
			mark options={solid}
			] coordinates {
				(0, 0.3934)
				(1, 0.2055)
				(2, 0.4726)
				(3, 0.3173)
				(4, 0.8162)
				(5, 0.4486)
				(6, 0.6914)
				(7, 0.2963)
				(8, 0.0400)
				(9, 0.0000)
			};
			
			\definecolor{darkpurple}{rgb}{0.4, 0.2, 0.6}
			
			\addplot[
			color=darkpurple,
			mark=star,
			mark size=2pt,
			line width=1pt,
			dotted,
			mark options={solid}
			] coordinates {
				(0, 1.0800)
				(1, 1.3485)
				(2, 1.0801)
				(3, 1.3999)
				(4, 1.1019)
				(5, 1.4634)
				(6, 1.1555)
				(7, 1.5291)
				(8, 1.9870)
				(9, 1.5969)
			};
			
			\addplot[
			color=darkpurple,
			mark=star,
			mark size=2pt,
			line width=1pt,
			solid,
			mark options={solid}
			] coordinates {
				(0, 0.3934)
				(1, 0.2524)
				(2, 0.5246)
				(3, 0.4134)
				(4, 1.0000)
				(5, 1.0000)
				(6, 1.0000)
				(7, 1.0000)
				(8, 1.0000)
				(9, 1.0000)
			};
			
		\end{axis}
	\end{tikzpicture}
	\caption{An idealised picture of two particles (dotted lines) and their moving boundaries (full lines) with $D_{12}=D_{21}=1$ and $D_{11}=D_{22}=0$. At $\tau_2$, the $2$nd particle (orange, circle markers) crosses the boundary and is absorbed, so the boundary of the $1$st particle (purple, star markers) settles at $1$ from then on. The $1$st particle is not absorbed, so the boundary of the $2$nd particle instead ends up at $0$ (as opposed to $1$).}
	\label{fig:two_particles}
\end{figure}
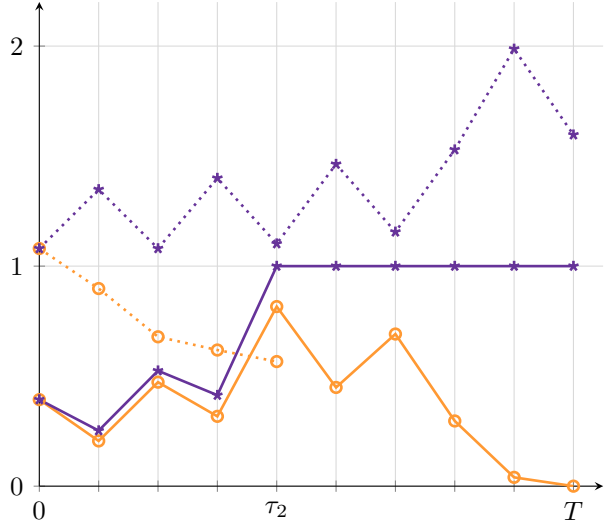

 Figure \ref{fig:two_particles} provides some intuition in a simple (idealised) setting of two particles in discrete time without self-interaction, i.e.\@ $D_{ii} = 0$ for $i = 1$, $2$. The first time step sees both boundaries recede, despite the $2$nd particle moving down: the $1$st particle going up pushes the $2$nd boundary down by enough to align with the $2$nd particle also having a lower probability of absorption, and so the 1st boundary recedes as well. At the next time step, both particles go down, so the boundaries inevitably increase. At time $\tau_2$, a large advance of the $2$nd boundary causes the $2$nd particle to be absorbed, and in turn the $1$st boundary now equals $D_{21}=1$. The $2$nd boundary continues to adjust to the evolution of the $1$st particle, but ultimately retracts towards $0$ (instead of $D_{12}=1$) as the absorption of the $1$st particle becomes improbable.

Let us now return to the general setting. Note that the conditional killing probabilities $\mathbb{P}(\tau_j \leq T| \mathcal{F}_t)$ are martingales. Hence, an equivalent way of expressing the problem is to search for $N$ moving boundaries that match the terminal conditions described above and are martingales. Clearly, their values must be resolved backwards in time, while the particles evolve forwards in time. In view of the martingale representation theorem, we can approach this as a system of forward-backward stochastic differential equations (FBSDEs) of the form
\begin{equation} \label{eq:fbsde_hetero}
	\d X^i_t = \sigma \, \d W^i_t, \qquad \d Y^i_t = Z^i_t \cdot \d \bf{W}_t,
\end{equation}
for $i = 1,~\ldots, N$, with initial conditions $X_0^i=\xi_i$, terminal conditions $Y^i_T = \bf{1}_{\{\tau_i \leq T\}}$, and a Brownian motion $\bf{W} = (W^1, \dots, W^N)$. The terminal conditions are determined by the hitting times
\begin{equation} \label{eq:hitting_times}
	\tau_i = \inf\biggl\{t \in [0, T]\define X^i_t \leq \sum_{j = 1}^N D_{ij} Y^j_t\biggr\}.
\end{equation}
with $\inf \emptyset = \infty$. The trajectory of the $i$th particle is then given by $X^{i}$ with absorption at $\tau_i$, and its associated moving boundary is the martingale $\sum_{j=1}^ND_{ij}Y^j$.

Let us emphasise that the coupling through the hitting times \eqref{eq:hitting_times} appearing in FBSDE \eqref{eq:fbsde_hetero} leads to a threefold singularity. Firstly, even if two realisations of a particle's trajectory and its associated moving boundary are uniformly close, the corresponding hitting times may be arbitrarily far apart. This reflects the singular nature of the absorption mechanism. Secondly, the terminal conditions involve indicator functions, which are discontinuous and, hence, singular in their own right. Finally, note that the diffusivity of the $i$th backward process degenerates upon the absorption of the $i$th particle (as illustrated in Figure \ref{fig:two_particles}). The focus of this work is the well-posedness of this singular FBSDE \eqref{eq:fbsde_hetero}. Existence and uniqueness of FBSDEs have been studied extensively, based on three general approaches: a contraction mapping in small time, first explored by Antonelli \cite{Antonelli1993}, the decoupling of the forward and backward equations proposed by Ma, Protter \& Yong \cite{MaProtterYong1994}, akin to the original treatment of BSDEs by Pardoux \& Peng \cite{PardouxPeng1990, PardouxPeng1992}, and finally the method of continuation by Hu \& Peng \cite{HuPeng1995} and Yong \cite{Yong1997}. Furthermore, Delarue \cite{delarue2002existence} later introduced a refinement of the contraction mapping approach, inductively extending a small time result to a global one based on a careful analysis of the decoupling field.

The contraction mapping approach places regularity assumptions on the coefficients and terminal condition, while the method of continuation requires monotonicity properties, neither of which apply due to the singular nature of FBSDE \eqref{eq:fbsde_hetero}. This suggests an analysis based on a decoupling field that separates the forward and backward equations. Decoupling exploits the Markovianity inherent to many FBSDEs, which allows one to express the backward state as a function of the forward state and time. This function, called the decoupling field, typically solves a partial differential equation (PDE). Thus, the FBSDE in question may be solved by taking a solution to the PDE describing the decoupling field, should it exist and be sufficiently regular, and using it to construct a solution to the FBSDE. After deriving some initial results for FBSDE \eqref{eq:fbsde_hetero} based on probabilistic techniques, this is also the main route that we shall take here.

Note, however, that the system \eqref{eq:fbsde_hetero} is obviously not Markovian in the forward states $(X^1_t, \dots, X^N_t)$ alone, as one must also keep track of absorptions. A possible solution is to append the absorption indicators $(\bf{1}_{\{\tau_1 \leq t\}}, \dots, \bf{1}_{\{\tau_N \leq t\}})$ to the forward state, but their digital nature does not harmonise with the PDE-based approach we are aiming for. Thus, we instead unravel what is usually a single decoupling field into a family of functions, each corresponding to a different configuration of the as yet unabsorbed particles. This family will satisfy a cascade of PDEs, with one level of the cascade determining the boundary and serving as a boundary condition for the next higher level. Since the decoupling fields vary in time, so do the boundaries they demarcate, leading to a free boundary problem. The techniques we develop to solve this cascade of PDEs and, thereby, FBSDE \eqref{eq:fbsde_hetero}, can be a starting point for the analysis of general classes of FBSDEs with coupling through absorption events. In particular, our methods extend to more general coefficients for both the forward and the backward equation. We stick to the minimalistic setting \eqref{eq:fbsde_hetero} throughout, since the essence of the analysis remains the same and the notation is already heavy enough as it is.

\subsection{Related Literature}

Our article sits at the intersection of two strands of literature that so far have seen little interplay: on the one hand, particle systems and Brownian motion in domains with moving boundaries and, on the other hand, FBSDEs with singular data.

We already mentioned the particle systems in Hambly, Ledger \& S{\o}jmark \cite{HamblyLedgerSojmark2019} and Nadtochiy \& Shkolnikov \cite{NadtochiyShkolnikov2020}, where the boundaries increase upon particles being absorbed instead of following martingale trajectories. The focus of that literature is the passage to and analysis of the associated mean-field limit as $N \to \infty$. The limit yields a probabilistic representation of the supercooled Stefan problem, which opened the door to global well-posedness results \cite{delarue_stefan_2022}.

Krylov \cite{krylov_brownian_boundary_2003, krylov_bm_boundary_2003} treats one-dimensional Brownian motion absorbed upon meeting the trajectory of another Brownian motion. From an analytical point, this can be recast as the study of a PDE with Dirichlet boundary condition at the path of a Brownian motion. Regularity properties are established for this PDE, focusing on the behaviour at the boundary. The reflected case was analysed by Burdzy \& Nualart \cite{burdzy_bm_2002}. In multiple dimensions, Burdzy, Chen \& Sylvester \cite{burdzy2004heat} study reflected Brownian motion in a moving domain with regular boundary, while they consider rougher boundaries in the one-dimensional case \cite{burdzy_rbm_2003}. In the former setting, a smooth fundamental solution to the forward equation associated with the reflected Brownian motion can be constructed, while in the latter situation two different types of singularities, heat atoms and heat singularities, may appear at the boundary. Little work exists on nonsmooth boundaries in multiple dimensions.

Without the moving boundary aspect, a recent preprint by Cardaliaguet, Jackson \& Souganidis \cite{cardaliaguet2025mean} studies the control of particle systems with absorption (at a fixed exogenous boundary). Similarly to how we deal with the absorption events in FBSDE \eqref{eq:fbsde_hetero}, using a cascade structure, they derive a cascade of PDEs (in the symmetric setting) to characterise the control problem's value function. The emphasis is on the convergence to the mean-field control problem that emerges in the limit as $N \to \infty$. Moving boundaries, the focal point and main challenge of the present work, are not considered.

Turning to forward-backward problems, we stress that FBSDEs with singular terminal conditions have been studied by Carmona, Delarue, Espinosa \& Touzi \cite{carmona2013singular} and Carmona \& Delarue \cite{carmona2013emissions}. They consider indicator functions of the event that the terminal state $X_T$ is above or below some value, thus revealing a similarity with our problem. While they take a decoupling approach, the analysis is of a quite different nature and relies on fine estimates for the PDEs associated to smooth approximations of the terminal conditions. Also, a significant part of their work explores what happens if the diffusion coefficient degenerates near the discontinuity of the terminal condition. In contrast, our central contribution lies in dealing with the hitting times. Identifying the associated PDE problem in this setting is a nontrivial task in itself, and the analysis of the resulting class of moving boundary problems with singular boundaries is outside the scope of standard methods.

In a recent treatment of FBSDEs with path-dependent coefficients by Hu, Ren \& Touzi \cite{HuRenTouzi2022}, a notion of decoupling fields on path space is studied, building on the theory of decoupling random fields for non-Markovian settings by Ma, Wu, Zhang \& Zhang \cite{MaWuZhangZhang2015}. Due to the singular nature of the coupling through hitting times in \eqref{eq:hitting_times}, our problem cannot be treated within these frameworks. Instead, the path-dependence arising from the hitting times is addressed through the moving boundary problem's cascade structure. The techniques we develop can be a starting point for the analysis of general classes of FBSDEs with coupling through absorption events. In particular, our methods extend to more general coefficients in \eqref{eq:fbsde_hetero}, but we stick with  \eqref{eq:fbsde_hetero} throughout, since the essence of the analysis remains the same and the notation is already heavy enough as it is.

Finally, we note that the singular FBSDEs studied in \cite{carmona2013singular, carmona2013emissions} were motivated by the mathematical analysis of emission markets, specifically the pricing of carbon allowances in cap-and-trade schemes. Our problem arises naturally in the study of contagion in financial networks. We explain how this occurs in the next section.

\subsection{Contagion in Financial Networks}\label{subsect:intro_contagion}

An active area of research in the finance and economics literature is the role that financial networks may play in propagating and amplifying shocks \cite{Acemoglu, elliott2014financial, jackson2024credit}. Following Acemoglu, Ozdaglar \& Tahbaz-Salehi \cite{Acemoglu}, we take as given a weighted directed network of $N$ banks represented by the liabilities $D_{ij}$ that bank $i$ is owed by bank $j$. Moreover, at a given point in time where these liabilities are due, we let $A^i$ denote the value of bank $i$'s external assets\footnote{In the notation of \cite{Acemoglu}, this is the sum $c_j+z_j+\zeta A$, where $c_j$ is a cash amount, $z_j$ is a return, and $A$ is the value of a `long-term' project which is realised today at a fraction $\zeta \in[0,1)$ of its value if this can allow the bank to avoid default by paying its liabilities in full.} and let $D^{i}$ denote bank $i$'s external liabilities (i.e., external to the network). We shall refer to the value of assets minus liabilities as the capital and denote this by $K^i$ for bank $i$. Whilst the original formulation looks quite different (see \cite[Definition 2]{Acemoglu}), the equilibrium model of contagion in \cite{Acemoglu} may be expressed as the fixed-point problem
\begin{align}
	K^i &= A^i + \sum_{j =1}^N \phi_{ij}(K^j) - D^i -  \sum_{j=1}^N D_{ji},\label{eq:acemoglu_fp} \\
	\phi_{ij}(K) &=D_{ij}\mathbf{1}_{\{ K > 0\} } + \frac{\bigl(K + \sum_{\ell=1}^N D_{\ell j} \bigr)_+}{\sum_{\ell=1}^N D_{\ell j}} D_{ij} \mathbf{1}_{\{ K \leq 0 \}} \label{eq:proportional}
\end{align}
for $i, j\in \{1, \dots, N\}$. The functions $\phi_{ij}$ specify the payments bank $i$ receives from bank $j$, with negative capital meaning that a bank is in default. Based on what bank $j$ has available to pay, after settling the external liabilities, \eqref{eq:proportional} enforces that each bank $i\neq j$ is paid an equal proportion of what it is owed. Another common rule, which we shall focus on below, is that a given proportion $R \in [0,1)$ is recovered upon default. That is, $\phi_{ij}(K)=D_{ij}\mathbf{1}_{K > 0} + R  D_{ij}\mathbf{1}_{K \leq 0} $ with $R$ encoding how costly defaults are.

The survey paper \cite{Glasserman} by Glasserman \& Young discusses how general problems of the form \eqref{eq:acemoglu_fp} can also model \emph{`situations where contagion is triggered by changes in market perceptions about the creditworthiness of particular institutions'}. This relies on a suitable choice of the functions $\phi_{ij}$ returning a \emph{`current mark-to-market value'} of the obligation $D_{ij}$. If one shocks $K^j$, then $\phi_{ij}(K^j)$ may be taken to decrease even if $K^j >0$ (to reflect lower creditworthiness), but then $K^i$ in turn decreases, and ultimately \emph{`these declines can lead to the outright default of some institutions, even though no one defaulted to begin with'}.

In the above, any shock is exogenous to the model and the implied timeline of events only refers to iterations towards a fixed point. Moreover, how to choose functions $\phi_{ij}$ that reflect perceptions about creditworthiness is not explored. To address this, consider a dynamic framework where obligations are due at a future time $T$, and let the external assets of bank $i$ evolve as $\mathrm{d}A^i_t=\sigma \mathrm{d}W^i_t$. Assume for simplicity that the risk-free interest rate is zero. If the banks fail as soon as their capital is negative, then the equations for the capital processes $K^i_t$ and the current mark-to-market values $\Phi_{ij}(t)$ become
\begin{align}
	K^i_t &= A^i_0 + \sigma W_t^i + \sum_{j = 1}^N \Phi_{ij}(t) - D^i - \sum_{j = 1}^N D_{ji},\label{eq:cond_dyn_hetero_0} \\
	\Phi_{ij}(t) &= D_{ij}\mathbb{P}(\tau_j > T \vert \mathcal{F}_t) + R D_{ij}\mathbb{P}(\tau_j\leq T \vert \mathcal{F}_t) \label{eq:valuation_functions_0}
\end{align}
with $\tau_j = \inf\{ t\in[0,T] \define K^j_t \leq 0\}$, for a given recovery rate $R\in[0,1)$. Thus, we obtain a dynamic counterpart of \eqref{eq:acemoglu_fp}--\eqref{eq:proportional} that directly models contagion through changing perceptions about the creditworthiness of the banks within the system. This is closely related to the work of Allen, Babus \& Carletti \cite{Allen} which highlighted that the updating of conditional default probabilities can be a key transmission channel for information contagion.

A version of \eqref{eq:cond_dyn_hetero_0}--\eqref{eq:valuation_functions_0} in discrete time and with discrete state space was recently studied by Feinstein \& S{\o}jmark \cite{feinstein_sojmark}. It was shown that there exist minimal and maximal solutions, and examples of nonuniqueness were given. Proposition \ref{prop:equivalence} below confirms that the problem \eqref{eq:cond_dyn_hetero_0}--\eqref{eq:valuation_functions_0} is equivalent to the FBSDE \eqref{eq:fbsde_hetero}. Remarkably, our analysis of \eqref{eq:fbsde_hetero} will allow us to recover uniqueness.


\subsection{Main Contributions and Structure of the Paper} \label{subsect:intro_heuristics}

We end the introduction with a brief outline of the paper and a heuristic explanation of our contributions. In Section \ref{sec:probabilistic}, we undertake a preliminary probabilistic analysis. Exploiting a monotonicity structure inherent in the problem (which should be distinguished from the monotonicity conditions formulated by Hu \& Peng \cite{HuPeng1995}), we are able to derive existence for FBSDE \eqref{eq:fbsde_hetero} through Tarski's fixed-point theorem. This result is complemented by some basic structural properties of solutions to FBSDE \eqref{eq:fbsde_hetero}. However, the analysis leaves several key questions unanswered, such as uniqueness, Markovianity, stability of the system with respect to initial conditions, and more. To address this, the remainder of the paper is concerned with an analytical approach based on decoupling FBSDE \eqref{eq:fbsde_hetero}.

In Section \ref{sec:moving_boundary_pde}, we introduce the cascade of moving boundary problems whose solution is intended to serve as a decoupling field. In a naive formulation of this, the boundary at a given level depends on the solution of that level itself. However, this can be untangled so that the moving boundary for a given level is determined by the solution of the preceding level. Establishing existence and uniqueness of classical solutions to this PDE problem is a delicate issue, since the moving boundary for each level has spatial kinks where its time regularity is challenging to ascertain. In addition, the temporal gradient of the boundary explodes at the discontinuity points of the terminal condition. These difficulties are resolved through a tailored analysis that exploits various structural properties of the boundary. 

Having established classical well-posedness of the moving boundary problem, we rigorously link it to FBSDE \eqref{eq:fbsde_hetero} in Section \ref{sec:verification}. That is, we construct a solution to FBSDE \eqref{eq:fbsde_hetero} based on the unique classical solution of the moving boundary PDE. Note that this does not imply uniqueness for the former, since there could in principle be solutions to FBSDE \eqref{eq:fbsde_hetero} that do not arise from the decoupling field. In fact, the discrete version of FBSDE \eqref{eq:fbsde_hetero} analysed by Feinstein \& S{\o}jmark \cite{feinstein_sojmark} exhibits nonuniqueness, so one may suspect the same to be true in our setting.

This supposition is refuted in Section \ref{sec:uniqueness_fbsde}, where we show that the solution stemming from the decoupling field is indeed the only one to FBSDE \eqref{eq:fbsde_hetero}. Due to the singular behaviour of the FBSDE, we cannot rely on any existing techniques such as contraction or monotonicity arguments. Instead, we exploit the following insight: if there are two distinct solutions to FBSDE \eqref{eq:fbsde_hetero}, then the absorption time of at least one of the particles will be different for the two solutions with positive probability. However, between these two absorption times, the Brownian motion driving the particle could drop to such a low level that a premature absorption of the particle is guaranteed, contradicting the definition of the later absorption time. A rigorous implementation of this proof strategy requires careful handling of the decoupling field.  

We conclude the paper with Section \ref{sec:mfl}, where we make some preliminary observations regarding a possible mean-field limit of the finite particle system as the number $N$ of particles is taken to infinity. We present a potential candidate, but we are not able to verify this as a limit at this stage. Surprisingly, the conjectured limit has an extremely simple structure compared to the finite system. Moreover, it again exhibits nonuniqueness similarly to the discrete setting discussed above.

\section{Probabilistic Analysis of FBSDE \texorpdfstring{\eqref{eq:fbsde_hetero}}{(EQ)}} \label{sec:probabilistic}

We begin our analysis with a brief and relatively simple probabilistic treatment of FBSDE \eqref{eq:fbsde_hetero}. This serves mainly to familiarise ourselves with the problem and establish some initial properties. In the subsequent sections, we shall then pursue an analytical approach to address the more subtle aspects of the problem. First, we fix the relevant notation and give a precise formulation of FBSDE \eqref{eq:fbsde_hetero}.

\subsection{Preliminaries and Precise Problem Formulation}

To simplify notation, we shall write $[N] = \{1, \dots, N\}$ throughout. We will be using the conventions $[\infty, T] = \varnothing$ and $\inf \varnothing = \infty$. Lastly, for $x$, $y \in \R^N$, we write $x \leq y$ if $x_i \leq y_i$ for $i \in [N]$. Fix a probability space $(\Omega, \F, \pr)$ which we take to support the $N$ independent Brownian motions $W^1,~\ldots , W^N$. For a given $\sigma$-algebra $\F_0 \subset \F$ which is independent of the Brownian motions, consider the filtration $\bb{F} = (\F_t)_{t \in[0,T]}$ defined by $\F_t = \F_0 \lor \sigma(W^i_s \define s \in [0, t],\, i \in [N])$.

Next, for $\bb{F}$-stopping times $\tau_1$, $\tau_2$ with values in $[0, T]$ such that $\tau_1 \leq \tau_2$, we let $\cal{T}_{\tau_1, \tau_2}$ denote the space of $\bb{F}$-stopping times $\varrho$ such that $\varrho_1 \leq \varrho \leq \varrho_2$. Furthermore, we let $\bb{H}^{2, d}_{\tau_1, \tau_2}$ be the space of $\R^d$-valued and $\bb{F}$-progressively measurable processes $Z = (Z_t)_{t \in [\tau_1, \tau_2]}$ with $\ev \int_{\tau_1}^{\tau_2} \lvert Z_t\rvert^2 \, \d t < \infty$. Likewise, we then let $\bb{S}^{2, d}_{\tau_1, \tau_2}$ be the space of $\R^d$-valued, $\bb{F}$-adapted, and continuous processes $X= (X_t)_{t \in [\tau_1, \tau_2]}$ with $\ev \sup_{t \in [\tau_1, \tau_2]} \lvert X_t\rvert^2 < \infty$. Throughout, we write $\bb{H}^{2, d} = \bb{H}^{2, d}_{0, T}$, $\bb{S}^{2, d} = \bb{S}^{2, d}_{0, T}$, $\bb{S}^2_{\tau_1, \tau_2} = \bb{S}^{2, 1}_{\tau_1, \tau_2}$, and $\bb{S}^2 = \bb{S}^{2, 1}$.

\begin{remark}
Given two stopping times $\varrho_1$, $\varrho_2 \in \cal{T}_{[0, T]}$ such that $\varrho_1 \leq \varrho_2$ and an $\F_{\varrho_2}$-measurable integrable random variable $\zeta$, we will often want to define a continuous process $\cal{E} = (\cal{E}_t)_{t \in [\varrho_1, \varrho_2]}$ such that $\cal{E}_{\tau}$ coincides a.s.\@ with $\ev[\zeta \vert \F_{\tau}]$ for all $\tau \in \cal{T}_{\varrho_1, \varrho_2}$. This is achieved by letting $\tilde{\cal{E}} = (\tilde{\cal{E}}_t)_{t \in [0, T]}$ be a continuous modification of $(\ev[\zeta \vert \F_t])_{t \in [0, T]}$ and setting
\begin{equation*}
	\cal{E}_t = \tilde{\cal{E}}_t \quad \text{for } t \in [\varrho_1, \varrho_2].
\end{equation*}
That is, for each $\omega  \in \Omega$, we set $\cal{E}_t(\omega) = \tilde{\cal{E}}_t(\omega)$ for $t \in [\varrho_1(\omega), \varrho_2(\omega)]$. As a shorthand for this construction, we write $\cal{E}_t = \ev[\zeta \vert \F_s] \vert_{s = t}$ for $t \in [\varrho_1, \varrho_2]$. What is meant is that the expression $\ev[\zeta \vert \F_s](\omega)$ for a deterministic time $s \in [0, T]$ is evaluated at $s = t$, where $t \in [\varrho_1(\omega), \varrho_2(\omega)]$. Plugging $t$ immediately into the conditional expectation does not make sense, since it depends on $\omega$. We circumvent this by the evaluation operation.
\end{remark}

Now, fix an arbitrary $N \times N$ weighted adjacency matrix $\mathbf{D}$ with entries $D_{ij} \geq 0$ for $i,j\in [N]$. Note that we do not necessarily impose $D_{ii} = 0$, meaning that there can be self-interaction. We shall introduce a generalisation of the system \eqref{eq:fbsde_hetero}, started at any given $\bb{F}$-stopping time $\varrho$ with values in $[0, T]$, any given $\F_{\varrho}$-measurable initial states $\xi_1$,~\ldots, $\xi_N$, and any given $\F_{\varrho}$-measurable subset $\chi \subset [N]$ of initially alive particles. We shall refer to $(\varrho, (\xi_i)_{i \in [N]}, \chi)$ as the \textit{initial data} of the problem. A solution to FBSDE \eqref{eq:fbsde_hetero} started from initial data $(\varrho, (\xi_j)_{j \in [N]}, \chi)$ is a tuple $(X^i, Y^i, Z^i)_{i \in [N]}$ with $(X^i, Y^i, Z^i) \in \bb{S}^2_{\varrho, T} \times \bb{S}^2_{\varrho, T} \times \bb{H}^{2, N}_{\varrho, T}$ such that $(X^i, Y^i, Z^i)_{i \in [N]}$ satisfies 
\begin{equation}
    X^i_t = \xi_i + \sigma(W^i_t - W^i_{\varrho}), \qquad Y^i_t = \bf{1}_{\{\tau_i \leq T\}} - \int_t^T Z^i_s \cdot \d \bf{W}_s
\end{equation}
for all $t \in [\varrho, T]$ and $i \in [N]$, where
\begin{equation*}
	\tau_i = \inf\biggl\{t \in [\varrho, T] \define X^i_t \leq \sum_{j = 1}^N D_{ij} Y^j_t\biggr\}
\end{equation*}
if $i \in \chi$ and $\tau_i = \varrho$ otherwise. The original problem amounts to $\chi=[N]$ and $\rho \equiv 0$, meaning that $ (X^i, Y^i, Z^i)_{i \in [N]}\in (\bb{S}^2 \times \bb{S}^2 \times \bb{H}^{2, N})^N$ satisfies \eqref{eq:fbsde_hetero} and \eqref{eq:hitting_times}.

Just as we write $\mathbf{W}$ for $(W^1,\ldots, W^N)$ in \eqref{eq:fbsde_hetero}, we shall use the boldface symbols $\bf{X}$ and $\bf{Y}$ to denote the vectors $(X^1, \dots, X^N)$ and $(Y^1, \dots, Y^N)$ when $(X^i, Y^i, Z^i)_{i \in [N]}$ is a solution to \eqref{eq:fbsde_hetero}. 

Recall the equilibrium model of contagion \eqref{eq:cond_dyn_hetero_0}--\eqref{eq:valuation_functions_0} derived in Section \ref{subsect:intro_contagion}. Setting $\xi_i := A^i_0 - D_i + \sum_{j = 1}^N (D_{ij} - D_{ji})$, the problem simplifies to
\begin{equation}\label{eq:cond_dyn_hetero}
K^i_t = \xi_i + \sigma W_t^i - (1-R) \sum_{j=1}^N D_{ij} \pr(\tau_j \leq T \vert \mathcal{F}_t)
\end{equation}
with $\tau_i = \inf\{ t\in[0,T] \define K^i_t \leq 0\}$ for $i\in [N]$. The following result makes precise that this is equivalent to our FBSDE problem \eqref{eq:fbsde_hetero}. Without loss of generality, we set $R=0$.

\begin{proposition} \label{prop:equivalence}
A family of processes $K^1$,~\ldots, $K^N \in \bb{S}^2$ satisfies \eqref{eq:cond_dyn_hetero} if and only if there exists a solution $(X^i, Y^i, Z^i)_{i \in [N]} \in (\bb{S}^2 \times \bb{S}^2 \times \bb{H}^{2, N})^N$ to FBSDE \eqref{eq:fbsde_hetero} such that $K^i_t = X^i_t - \sum_{j = 1}^N D_{ij} Y^j_t$ for all $i \in [N]$.
\end{proposition}

\begin{proof} Suppose first that $K^1$,~\ldots, $K^N \in \bb{S}^2$ form a solution to \eqref{eq:cond_dyn_hetero}. By the martingale representation theorem (see e.g.\@ \cite[Theorem 2.5.2]{zhang_bsde_2017}), there exists an $N$-dimensional process
	$Z^i = (Z^{ij})_{j \in [N]} \in \bb{H}^{2, N}$ such that
	\begin{equation*}
		\bf{1}_{\tau_i \leq T}  = \pr(\tau_i \leq T \vert \F_t) + \int_t^T Z^i_s \cdot \d \bf{W}_s
	\end{equation*}
	for $t \in [0, T]$. Hence, if we set $Y^i_t = \pr(\tau_i \leq T \vert \F_t)$, it follows that $(X^i, Y^i, Z^i)_{i \in [N]}$ satisfies the FBSDE \eqref{eq:fbsde_hetero}. Conversely, if we have a solution $(X^i, Y^i, Z^i)_{i \in [N]}$ to \eqref{eq:fbsde_hetero}, then defining $K^1$,~\ldots, $K^N \in \bb{S}^2$ by
	\begin{equation*}
		K^i_t = \xi_i + \sigma W^i_t - \sum_{j = 1}^N D_{ij} \biggl(Y^j_0 + \int_0^t Z^j_s \cdot \d \bf{W}_s\biggr) = \xi_i + \sigma W^i_t - \sum_{j = 1}^N D_{ij} Y^j_t
	\end{equation*}
	for $i \in [N]$ and using that $Y^j_t = \ev[Y^j_T \vert \F_t] = \pr(\tau_j \leq T \vert \F_t)$, we find that $K^1$,~\ldots, $K^N$ follow the dynamics from \eqref{eq:cond_dyn_hetero}. This completes the proof.
\end{proof}

\subsection{Existence of FBSDE \texorpdfstring{\eqref{eq:fbsde_hetero}}{(EQ)} and Basic Properties}

We begin with a definition of minimal and maximal solutions for FBSDE \eqref{eq:fbsde_hetero}.

\begin{definition}
We call a solution $(X^i, Y^i, Z^i)_{i \in [N]}$ to FBSDE \eqref{eq:fbsde_hetero} started from initial data $(\varrho, (\xi_i)_{i \in [N]}, \chi)$ \textit{minimal} (\textit{maximal}) if for any other solution $(\tilde{X}^i, \tilde{Y}^i, \tilde{Z}^i)_{i \in [N]}$ started from $(\varrho, (\xi_i)_{i \in [N]}, \chi)$ it holds that a.s.\@ $Y^i_t \leq \tilde{Y}^i_t$ ($Y^i_t \geq \tilde{Y}^i_t$) for all $t \in [\varrho, T]$ and $i \in [N]$.
\end{definition}

Clearly, if they exist, minimal and maximal solutions are by definition unique. 

\begin{remark} \label{rem:min_max}
Note that for any two solutions $(X^i, Y^i, Z^i)_{i \in [N]}$ and $(\tilde{X}^i, \tilde{Y}^i, \tilde{Z}^i)_{i \in [N]}$ of FBSDE \eqref{eq:fbsde_hetero}, $Y^i_T \leq \tilde{Y}^i_T$ a.s.\@ implies that for all $\tau \in \cal{T}_{\varrho, T}$, we have a.s.\@ that
\begin{equation*}
    Y^i_{\tau} = \ev[Y^i_T \vert \F_{\tau}] \leq \ev[\tilde{Y}^i_T \vert \F_{\tau}] \leq \tilde{Y}^i_{\tau}.
\end{equation*}
Since both $Y^i$ and $\tilde{Y}^i$ have continuous trajectories, we obtain that a.s.\@ $Y^i_t \leq \tilde{Y}^i_t$ for all $t \in [\varrho, T]$. In other words, a solution $(X^i, Y^i, Z^i)_{i \in [N]}$ to FBSDE \eqref{eq:fbsde_hetero} is already minimal (maximal) if $Y^i_T \leq \tilde{Y}^i_T$ a.s.\@ ($Y^i_T \geq \tilde{Y}^i_T$ a.s.\@) for any other solution $(\tilde{X}^i, \tilde{Y}^i, \tilde{Z}^i)_{i \in [N]}$.
\end{remark}

We have the following existence and comparison result for minimal and maximal solutions. Its proof exploits the monotonicity structure inherent in FBSDE \eqref{eq:fbsde_hetero}, which allows for the application of Tarski's fixed-point theorem. The uniqueness question is not addressed by this approach, but it turns out that we shall be able to tackle this via the analytical investigations in the next section (see\@ Theorem \ref{thm:uniqueness})

\begin{theorem} \label{thm:existence}
For any initial data, there exists a minimal and a maximal solution to FBSDE \eqref{eq:fbsde_hetero}.  Furthermore, if $(X^{k, i}, Y^{k, i}, Z^{k, i})_{i \in [N]}$ is the minimal (maximal) solution to FBSDE \eqref{eq:fbsde_hetero} with initial data $(\varrho_k, (\xi^k_i)_{i \in [N]}, \chi_k)$, $k = 1$, $2$, such that a.s.\@ $\varrho_1 \leq \varrho_2$, $X^{1, i}_{\varrho_2} \leq \xi^2_i$ for $i \in [N]$, and $\chi_1 \subset \chi_2$, then a.s.\@ $Y^{1, i}_t \geq Y^{2, i}_t$ for $t \in [\varrho_2, T]$.
\end{theorem}

\begin{proof}
\textit{Existence}: As discussed above the statement of the theorem, we intend to apply Tarski's fixed-point theorem. This result guarantees the existence of a least and greatest fixed point for monotonic maps on complete lattices. We shall first introduce a complete lattice that is suitable for our purposes. Denote by $L^0(\mathcal{F}_T;\{0,1\}^N)$ the set of all $\mathcal{F}_T$-measurable random vectors $\zeta$ with values in $\{0, 1\}$. Next, we introduce a partial ordering ``$\leq$'' on $L^0(\mathcal{F}_T;\{0,1\}^N)$ given by almost sure component-wise domination. That is, $\zeta \leq \eta$ for $\zeta$, $\eta \in L^0(\mathcal{F}_T; \{0,1\}^N)$ if $\zeta(\omega) \leq \eta(\omega)$ for a.e.\@ $\omega \in \Omega$. Note that, for any collection $(\zeta^i)_{i \in I}$ in $ L^0(\mathcal{F}_T; \{0,1\}^N)$ for an arbitrary index set $I$, we have that $\zeta = \esssup_{i \in I} \zeta^i $ is again in $L^0(\mathcal{F}_T; \{0,1\}^N)$, so $\zeta$ yields a least upper bound for the set $(\zeta^i)_{i \in I}$ under the partial order defined above. Analogously, the essential infimum yields a greatest lower bound. Consequently, $L^0(\mathcal{F}_T;\{0,1\}^N)$ is a complete lattice under this partial order.

Next, let us fix initial data $(\varrho, (\xi_i)_{i \in [N]}, \chi)$ and define a mapping $\Psi \define L^0(\mathcal{F}_T; \{0,1\}^N) \rightarrow L^0(\mathcal{F}_T; \{0,1\}^N)$ for $\zeta \in L^0(\mathcal{F}_T; \{0,1\}^N)$ by
\begin{equation} \label{eq:fp_map}
    \Psi(\zeta)= \bigl(\Psi_1(\zeta),\ldots, \Psi_N(\zeta)\bigr) = \bigl(\mathbf{1}_{\{\tau_1(\zeta) \leq T\}},\ldots, \mathbf{1}_{\{\tau_N(\zeta)\leq T\}}\bigr),
\end{equation}
where
\begin{equation*}
    \tau_i(\zeta) = \inf\biggl\{t \in [\varrho, T] \define X^i_t \leq \sum_{j = 1}^N D_{ij} Y^j_t(\zeta) \biggr\}
\end{equation*}
if in $i \in \chi$ and $\tau_i(\zeta) = \varrho$ otherwise. The processes $X^i = (X_t)_{t \in [\varrho, T]}$ and $Y^i(\zeta) = (Y^i_t(\zeta)_{t \in [\varrho, T]})$ are given by $X^i_t = \xi_i + \sigma(W^i_t - W^i_{\varrho})$ for $t \in [\varrho, T]$ and $Y^j_t(\zeta) = \ev[\zeta_j \vert \F_s]\vert_{s = t}$ for $t \in [\varrho, T]$. Observe that if $\zeta \leq \eta$ for $\zeta$, $\eta \in L^0(\mathcal{F}_T, \{0,1\}^N)$, then a.s.\@ $Y^j_t(\zeta) \leq Y^j_t(\eta)$ for all $t \in [\varrho, T]$. Since $D_{ij} \geq 0$, it follows that $\sum_{j = 1}^N D_{ij} Y^j_t(\zeta) \leq \sum_{j = 1}^N D_{ij} Y^j_t(\eta)$. Consequently, $\tau_i(\zeta) \leq  \tau_i(\eta)$, which yields that $\Psi(\zeta) \leq \Psi(\eta)$. That is, $\Psi$ is a monotonic mapping for the partial order on $L^0(\mathcal{F}_T; \{0,1\}^N)$ and, hence, Tarski's fixed-point theorem provides a least and greatest fixed point of $\Psi$ in $L^0(\mathcal{F}_T; \{0,1\}^N)$. For any such fixed point $\zeta \in L^0(\mathcal{F}_T; \{0,1\}^N)$, it holds that $\zeta_i = \bf{1}_{\{\tau_i(\zeta) \leq T\}}$. Hence, setting $\tilde{Y}^i_t = \ev[\zeta_i \vert \F_s]\vert_{s = t} = \pr(\tau_i(\zeta) \leq T \vert \F_s)\vert_{s = t}$ for $t \in [\varrho, T]$ and obtaining $Z^i \in \bb{H}^{2, N}_{\varrho, T}$ from the martingale representation theorem such that
\begin{equation*}
    Y^i_t = \bf{1}_{\{\tau_i(\zeta) \leq T\}} - \int_t^T Z^i_s \cdot \d \bf{W}_s
\end{equation*}
for $t \in [\varrho, T]$, it follows that $(X^i, Y^i, Z^i)_{i \in [N]}$ is a solution to FBSDE \eqref{eq:fbsde_hetero}. Clearly, the least and greatest fixed point of $\Psi$ correspond to the minimal and maximal solution of FBSDE \eqref{eq:fbsde_hetero}, respectively.

\textit{Comparison}: Let the initial data $(\varrho_k, (\xi^k_i)_{i \in [N]}, \chi_k)$, $k = 1$, $2$, be as in the statement of the theorem and denote by $\Psi^k$ and $\tau^k_i(\zeta)$, $i \in [N]$, $\zeta \in L^0(\mathcal{F}_T; \{0,1\}^N)$, the corresponding maps $L^0(\mathcal{F}_T; \{0,1\}^N) \to L^0(\mathcal{F}_T; \{0,1\}^N)$ and stopping times constructed in \eqref{eq:fp_map} and below. Owing to the assumed relationship between the two initial data, it holds for $\zeta \in L^0(\mathcal{F}_T; \{0,1\}^N)$ that
\begin{equation*}
    \inf\biggl\{t \in [\varrho_1, T] \define X^{1, i}_t \leq \sum_{j = 1}^N D_{ij} Y^j_t(\zeta)\biggr\} \leq \inf\biggl\{t \in [\varrho_2, T] \define X^{2, i}_t \leq \sum_{j = 1}^N D_{ij} Y^j_t(\zeta)\biggr\},
\end{equation*}
where $X^{k, i}_t = \xi^k_i + \sigma(W^i_t - W^i_{\varrho_k})$ for $t \in [\varrho_k, T]$. Since, furthermore, $\chi_1 \subset \chi_2$, we obtain that $\tau^1_i(\zeta) \leq \tau^2_i(\zeta)$. From this, we conclude that $\Psi^1(\zeta) \geq \Psi^2(\zeta)$ for all $\zeta \in L^0(\mathcal{F}_T; \{0,1\}^N)$. Now, the least and greatest fixed point $\zeta^{k, -}$ and $\zeta^{k, +}$ of $\Psi^k$ provided by Tarski's fixed-point theorem are simply the greatest lower bound and least upper bound of the set
\begin{equation*}
    \bigl\{\zeta \in L^0(\mathcal{F}_T; \{0,1\}^N) \define \zeta \leq \Psi^k(\zeta)\bigr\} \quad \text{and} \quad \bigl\{\zeta \in L^0(\mathcal{F}_T; \{0,1\}^N) \define \zeta \geq \Psi^k(\zeta)\bigr\},
\end{equation*}
respectively. However, it follows from the inequality $\Psi^1(\zeta) \geq \Psi^2(\zeta)$ that
\begin{equation*}
    \bigl\{\zeta \in L^0(\mathcal{F}_T; \{0,1\}^N) \define \zeta \leq \Psi^2(\zeta)\bigr\} \subset \bigl\{\zeta \in L^0(\mathcal{F}_T; \{0,1\}^N) \define \zeta \leq \Psi^1(\zeta)\bigr\}
\end{equation*}
and 
\begin{equation*}
    \bigl\{\zeta \in L^0(\mathcal{F}_T; \{0,1\}^N) \define \zeta \geq \Psi^1(\zeta)\bigr\} \subset \bigl\{\zeta \in L^0(\mathcal{F}_T; \{0,1\}^N) \define \zeta \geq \Psi^2(\zeta)\bigr\}.
\end{equation*}
Consequently, the least and greatest fixed points satisfy $\zeta^{1, -} \geq \zeta^{2, -}$ and $\zeta^{1, +} \geq \zeta^{2, +}$. Let $(X^{k, \pm, i}, Y^{k, \pm, i}, Z^{k, \pm, i})_{i \in [N]}$ denote the solutions to FBSDE \eqref{eq:fbsde_hetero} associated with the fixed point $\zeta^{k, \pm}$. As mentioned earlier, $(X^{k, -, i}, Y^{k, -, i}, Z^{k, -, i})_{i \in [N]}$ is the minimal solution with initial data $(\varrho_k, (\xi^k_i)_{i \in [N]}, \chi_k)$, while $(X^{k, +, i}, Y^{k, +, i}, Z^{k, +, i})_{i \in [N]}$ is maximal. From the relationship between the least and greatest fixed points, we can conclude that for every $\tau \in \cal{T}_{\varrho_2, T}$, we have a.s.\@ that
\begin{equation*}
    Y^{1, \pm, i}_{\tau} = \ev\bigl[\zeta^{1, \pm}_i \big\vert \F_{\tau}\bigr] \geq \ev\bigl[\zeta^{2, \pm}_i \big\vert \F_{\tau}\bigr] = Y^{2, \pm, i}_{\tau}.
\end{equation*}
Since $Y^{k, \pm, i}$ has continuous trajectories, we get that a.s.\@ $Y^{1, \pm, i}_t \geq Y^{2, \pm, i}_t$ for $t \in [\varrho_2, T]$.
\end{proof}

By applying the comparison result from Theorem \ref{thm:existence}, we can deduce a flow property for maximal solutions. The particular argument used in its proof exploits the defining characteristic of maximal solutions and, as such, does not apply to minimal solutions. 

\begin{corollary} \label{cor:flow_property}
Let $(X^{k, i}, Y^{k, i}, Z^{k, i})_{i \in [N]}$ be the maximal solution to FBSDE \eqref{eq:fbsde_hetero} with initial data $(\varrho_k, (\xi^k_i)_{i \in [N]}, \chi_k)$, $k = 1$, $2$, such that a.s.\@ $\varrho_1 \leq \varrho_2$, $\xi^2_i = X^{1, i}_{\varrho_2}$ for $i \in [N]$, and $\chi_2 = \{i \in [N] \define \tau^1_i > \varrho_2\}$, where $\tau^1_i$ are the killing times of the system $(X^{1, i}, Y^{1, i}, Z^{1, i})_{i \in [N]}$. Then a.s.\@
\begin{equation*}
    (X^{1, i}_t, Y^{1, i}_t, Z^{1, i}_t)_{i \in [N]} = (X^{2, i}_t, Y^{2, i}_t, Z^{2, i}_t)_{i \in [N]}
\end{equation*}
for all $t \in [\varrho_2, T]$.
\end{corollary}

\begin{proof}
Let $(X^{k, i}, Y^{k, i}, Z^{k, i})_{i \in [N]}$, $k = 1$, $2$, be as in the statement of the corollary. By Theorem \ref{thm:existence}, we have that a.s.\@ $Y^{1, i}_t \geq Y^{2, i}_t$ for $t \in [\varrho_2, T]$ and $i \in [N]$. However, it is easy to verify that $(X^{1, i}, Y^{1, i}, Z^{1, i})_{i \in [N]}$ when restricted to the interval $[\varrho_2, T]$ is a solution to FBSDE \eqref{eq:fbsde_hetero} with the same initial data $(\varrho_2, (\xi^2_i)_{i \in [N]}, \chi_2)$ as $(X^{2, i}, Y^{2, i}, Z^{2, i})_{i \in [N]}$. Since the latter solution is the maximal one, we obtain reverse inequality $Y^{1, i}_t \leq Y^{2, i}_t$ as well. Hence, $Y^{1, i}$ and $Y^{2, i}$ coincide for each $i \in [N]$. From this and the uniqueness part of the martingale representation theorem, we get that $(X^{1, i}, Y^{1, i}, Z^{1, i})_{i \in [N]}$ and $(X^{2, i}, Y^{2, i}, Z^{2, i})_{i \in [N]}$ agree on $[\varrho_2, T]$.
\end{proof}

It turns out that the maximal solution can be obtained as the limit of a monotonic sequence obtained by repeated application of the mapping $\Psi$ defined in and below \eqref{eq:fp_map}. By carefully examining this procedure, we can see that the maximal solution can be constructed solely from the initial data $D = (\varrho, (\xi_i)_{i \in [N]}, \chi)$ and the increments $(\bf{W}_t - \bf{W}_{\varrho})_{t \in [\varrho, T]}$ of the Brownian motion after the initial time $\varrho$. That is, the maximal solution is strong in a probabilistic sense. 

Define the filtration $\bb{F}^D = (\cal{F}^D_t)_{t \in [0, T]}$ by
\begin{equation*}
    \cal{F}^D_t = \sigma\Bigl(\bigl\{\varrho \leq s,\, (\xi_i)_{i \in [N]} \in A, \, \chi = B\bigr\}, \bf{W}_{s \lor \varrho} - 
    \bf{W}_{\varrho} \define s \in [0, t],\, A \in \cal{B}(\R^N),\, B \in \P_N\Bigr)
\end{equation*}
for $t \in [0, T]$, where $\P_N$ is the power set of $[N]$. We construct a sequence $(Y^{n, i}, Z^{n, i})_{i \in [N]}$, $n \geq 0$, with $Y^{n, i} \in \bb{S}^2_{\varrho, T}$ and $Z^{n, i} \in \bb{H}^{2, N}_{\varrho, T}$ as follows: first, we $Y^{0, i}_t = 1$ and $Z^{0, i}_t = 0$ for $t \in [\varrho, T]$. Now, assuming that $(Y^{n, i}, Z^{n, i})_{i \in [N]}$ is given for some $n \geq 0$, set
\begin{equation*}
    \tau^n_i = \inf\biggl\{t \in [\varrho, T] \define X^i_t \leq \sum_{j = 1}^N D_{ij} Y^{n, j}_t\biggr\}
\end{equation*}
if $i \in \chi$ and $\tau^n_i = \varrho$ otherwise. Then, we define $(Y^{n + 1, i})_{i \in [N]}$ by
\begin{equation*}
    Y^{n + 1, i}_t = \pr(\tau^n_i \leq T \vert \cal{F}^D_s)\vert_{s = t}
\end{equation*}
and obtain the $\bb{F}^D$-progressively measurable processes $(Z^{n + 1, i})_{i \in [N]}$ with $Z^{n + 1, i} \in \bb{H}^{2, N}_{\varrho, T}$ from the martingale representation theorem, such that
\begin{equation*}
    Y^{n + 1, i}_t = \bf{1}_{\{\tau^n_i \leq T\}} - \int_t^T Z^{n + 1, i}_s \cdot \d \bf{W}_s
\end{equation*}
for $t \in [\varrho, T]$. To see that the martingale representation theorem applies to the filtration $\bb{F}^D$ in the desired way, it is instructive to slightly change the point of view. Let us define the $\sigma$-algebra $\tilde{\cal{F}}^D_0 = \sigma(\varrho, (\xi_i)_{i \in [N]}, \chi)$ and the Brownian motions $\tilde{W}^i = (\tilde{W}^i_t)_{t \geq 0}$, $i \in [N]$, by $\tilde{W}^i_t = W^i_{\varrho + t} - W^i_{\varrho}$ for $t \geq 0$. Clearly, $\tilde{\bf{W}} = (\tilde{W}^1, \dots, \tilde{W}^N)$ is independent of $\tilde{\cal{F}}^D_0$. Next, we let the filtration $\tilde{\bb{F}}^D = (\tilde{\cal{F}}^D)_{t \geq 0}$ be given by
\begin{equation*}
    \tilde{\cal{F}}^D_t = \tilde{\cal{F}}^D_0 \lor \sigma\bigl(\tilde{W}^i_s \define s \in [0, t],\, i \in [N]\bigr)
\end{equation*}
for $t \geq 0$. Then, we have $\tilde{\cal{F}}^D_t = \cal{F}^D_{\varrho + t}$, so that $(Y^{n + 1, i}_{(\varrho + t) \land T})_{t \in [0, T]}$ is an $\tilde{\bb{F}}^D$-martingale. Since the martingale representation theorem applies to the filtration $\tilde{\bb{F}}^D$, we can find $\R^N$-valued $\tilde{\bb{F}}^D$-progressively measurable processes $\tilde{Z}^i = (\tilde{Z}^i_t)_{t \in [0, T]}$ such that $\ev\int_0^T \lvert \tilde{Z}^i_t\rvert^2 \, \d t < \infty$ and
\begin{equation*}
    Y^{n + 1, i}_{\varrho + t} = Y^{n + 1, i}_T - \int_t^T \tilde{Z}^i_s \cdot \d \tilde{\bf{W}}_s = \bf{1}_{\{\tau^n_i \leq T\}} - \int_t^{T - \varrho} \tilde{Z}^i_s \cdot \d \bf{W}_s
\end{equation*}
for $t \in [0, T - \varrho]$. Hence, $Z^{n + 1, i} \in \bb{H}^{2, N}_{\varrho, T}$ defined by $Z^{n + 1, i}_t = \tilde{Z}^i_{t - \varrho}$ for $t \in [\varrho, T]$ is the desired $\bb{F}^D$-progressively measurable process.

We say that a sequence $(x^n)_{n \geq 1}$ in $\R^N$ is \textit{nonincreasing} (\textit{nondecreasing}) if $x^{n + 1} \leq x^n$ ($x^n \geq x^{n + 1}$) for all $n \geq 1$.

\begin{proposition} \label{prop:maximal}
The sequence $(Y^{n, i}_t)_{i \in [N]}$, $n \geq 1$, is a.s.\@ nonincreasing for all $t \in [\varrho, T]$ and $(X^i, Y^{n, i}, Z^{n, i})_{i \in [N]}$ converges to the maximal solution $(X^i, Y^i, Z^i)_{i \in [N]}$ of FBSDE \eqref{eq:fbsde_hetero} in $(\bb{S}^2_{\varrho, T} \times \bb{S}^2_{\varrho, T} \times \bb{H}^{2, N}_{\varrho, T})^N$ as $n \to \infty$. In particular, $(X^i, Y^i, Z^i)_{i \in [N]}$ is $\bb{F}^D$-progressively measurable.
\end{proposition}


\begin{proof} 
We show by induction that $(Y^{n, i})_{i \in [N]}$, $n \geq 0$, is nonincreasing and bounded from below by the maximal solution $(Y^i)_{i \in [N]}$ of FBSDE \eqref{eq:fbsde_hetero}. Since
\begin{equation*}
    Y^{1, i}_t = \pr(\tau^0_i \leq T \vert \cal{F}^D_s)\vert_{s = t} \leq 1 = Y^{0, 1}_t
\end{equation*}
and $Y^{0, 1}_t = 1 \geq Y^i_t$ for all $t \in [\varrho, T]$, nonincreasingness and boundedness from below by the maximal solution is clear for $n = 0$. Now assume the claim holds for $n - 1$ for some $n \geq 1$. Then by the induction hypothesis,
\begin{align*}
    \tau^{n + 1}_i = \inf\biggl\{t \in [\varrho, T] \define X^i_t \leq \sum_{j = 1}^N D_{ij} Y^{n, j}_t \biggr\} \leq \inf\biggl\{t \in [\varrho, T] \define X^i_t \leq \sum_{j = 1}^N D_{ij} Y^{n - 1, j}_t \biggr\} = \tau^n_i
\end{align*}
if $i \in \chi$ and $\tau^{n + 1}_i = \varrho = \tau^n_i$ otherwise. Hence, $Y^{n + 1, i}_t = \pr(\tau^{n + 1}_i \leq T \vert \cal{F}^D_s)\vert_{s = t} \leq \pr(\tau^n_i \leq T \vert \cal{F}^D_s)\vert_{s = t} = Y^{n, i}_t$ for $t \in [\varrho, T]$. Similarly, using that $Y^{n, i}_t \geq Y^i_t$, we deduce that $\tau^{n + 1}_i \leq \tau_i$, whereby $Y^{n + 1, i}_t \geq Y^i_t$ for $t \in [\varrho, T]$. This concludes the induction. 

Next, we want to pass to the limit in the sequence $(Y^{n, i})_{i \in [N]}$. First, since the sequence of stopping time $(\tau^n_i)_{n \geq 1}$ is nondecreasing and bounded from above by $\tau_i$, it converges to a limit $\tilde{\tau}_i \leq \tau_i$ almost surely. Let us define the process $(\tilde{Y}^i_t)_{t \in [\varrho, T]}$ for $i \in [N]$, by $\tilde{Y}^i_t = \pr(\tilde{\tau}_i \leq T \vert \cal{F}^D_s)\vert_{s = t}$. It follows from the monotone convergence theorem that
\begin{equation*}
    \tilde{Y}^i_{\tau} = \pr(\tilde{\tau}_i \leq T \vert \F^D_{\tau}) = \lim_{n \to \infty} \pr(\tau^n_i \leq T \vert \F^D_{\tau}) = \lim_{n \to \infty} Y^{n, i}_{\tau} \geq Y^i_{\tau}
\end{equation*}
a.s.\@ for all $\tau \in \cal{T}_{\varrho, T}$. Moreover, by Doob's martingale inequality and the convergence $Y_T^{n,i} \rightarrow \tilde{Y}^i_T$, we have 
\begin{equation*}
    \ev \sup_{t \in [\varrho, T]} \lvert Y^{n, i}_t - \tilde{Y}^i_t\rvert^2 \leq 4 \ev \lvert Y^{n, i}_T - \tilde{Y}^i_T\rvert^2 \to 0.
\end{equation*}
Along a subsequence, which for simplicity we denote by the same index, this gives uniform convergence almost surely. Consequently, 
\begin{equation*}
    0 \geq \lim_{n \to \infty} \biggl(X^i_{\tau^n_i} - \sum_{j = 1}^N D_{ij} Y^{n, j}_{\tau^n_i}\biggr) = X^i_{\tilde{\tau}_i} - \sum_{j = 1}^N D_{ij} \tilde{Y}^j_{\tilde{\tau}_i}
\end{equation*}
so that $\tilde{\tau}_i$ is lower bounded by the first hitting time of $X^i_t - \sum_{j = 1}^N D_{ij} \tilde{Y}^j_t$ of $(-\infty, 0]$ if $i \in \chi$. The reverse inequality is obvious since $\tilde{Y}^j$ lies below $Y^{n, j}$ for $j \in [N]$. Therefore, if $i \in \chi$, then $\tilde{\tau}_i$ is the first time in $[\varrho, T]$ that $X^i_t - \alpha \sum_{j = 1}^N D_{ij} \tilde{Y}^j_t$ visits $(-\infty, 0]$. If $i \notin \chi$, then we trivially have $\tilde{\tau}_i = \varrho$. Hence, $(X^i, \tilde{Y}^i, \tilde{Z}^i)_{i \in [N]}$, with $\bb{F}^D$-progressively measurable $\tilde{Z}^i \in \bb{H}^{2, N}_{\varrho, T}$ such that
\begin{equation*}
    \tilde{Y}^i_t = \bf{1}_{\{\tilde{\tau}_i \leq T\}} - \int_t^T \tilde{Z}^i_s \cdot \d W_s
\end{equation*}
for $t \in [\varrho, T]$, provided by the martingale representation theorem, is a solution to FBSDE \eqref{eq:fbsde_hetero}. Note that $\tilde{Y}^i_t \geq Y^i_t$ for $t \in [\varrho, T]$ by construction, so the maximality of $(X^i, Y^i, Z^i)_{i \in [N]}$ immediately implies that $(X^i, \tilde{Y}^i, \tilde{Z}^i)_{i\ in [N]}$ equals the maximal solution $(X^i, Y^i, Z^i)_{i \in [N]}$, as desired.
\end{proof}

Note that the proof of Proposition \ref{prop:maximal} cannot be straightforwardly adjusted to obtain an approximation of the minimal solution from below. Indeed, initiating the sequence $(Y^{n, i}, Z^{n, i})_{i \in [N]}$, $n \geq 0$, with $Y^{0, i}_t = 0$ and $Z^{0, i}_t = 0$ for $t \in [\varrho, T]$, the corresponding sequence of stopping times $\tau^n_i$, $n \geq 1$, would be nonincreasing instead of nondecreasing, so that one cannot readily deduce the a.s.\@ convergence $Y^{n, i}_T = \bf{1}_{\tau^n_i \leq T} \to \bf{1}_{\tilde{\tau}_i \leq T} = \tilde{Y}^i_T$. Furthermore, even if this convergence were to hold, it still would not immediately follow that $\tilde{\tau}_i$ is the first hitting time of $X^i_t - \sum_{j = 1}^N D_{ij} \tilde{Y}^j_t$ on $(-\infty, 0]$.

As mentioned above, we will later see that FBSDE \eqref{eq:fbsde_hetero} in fact admits a unique solution, so that the minimal and maximal solutions from Theorem \ref{thm:existence} coincide. Hence, the present statement should be understood as saying that the unique solution of FBSDE \eqref{eq:fbsde_hetero} can be straightforwardly obtained as the monotonic limit of a decreasing sequence, while an approximation from below is a more subtle question.

For the remainder of the paper, we shall now switch gears and consider FBSDE \eqref{eq:fbsde_hetero} from an analytic point of view. This turns out to be a much more fruitful avenue for a detailed understanding of the problem.

\section{Moving Boundary PDE for FBSDE \texorpdfstring{\eqref{eq:fbsde_hetero}}{(EQ)}} \label{sec:moving_boundary_pde}

The goal of this section is to identify and analyse a system of PDEs with moving boundaries whose solution will serve as a decoupling field for FBSDE \eqref{eq:fbsde_hetero}. The system of PDEs describes the conditional killing probabilities $Y^i_t = \pr(\tau_i \leq T \vert \F_t)$ for a solution to FBSDE \eqref{eq:fbsde_hetero} for different configurations of the particle system. Each configuration $I \subset [N]$ corresponds to a different subset of particles initially assumed to be alive. For any configuration $I$ and any particle $i \in [N]$, we are given a function $v^{I, i} \define [0, T] \times \R^I \to \R$ such that if the system is started from initial states $x = (x_j)_{j \in I}$ at time $t \in [0, T]$, then $Y^i_t = v^{I, i}_t(x)$. Note that the decoupling field $v^{I, i}$ only depends on the states of the initially alive particle $j \in I$, since killed particles do not contribute to the evolution of $Y^i$. The domain of the PDE for $v^{I, i}$ for a given configuration $I \subset [N]$ consists of the set of states $x \in \R^I$ for which all particles remain alive. The boundary of this region changes over time, resulting in a moving boundary problem. On the boundary of the domain, at least one of the particles $j \in I$ is killed, yielding a configuration $I \setminus \{j\}$ with one fewer particle. This establishes a connection between the decoupling fields amongst different configurations. 

To rigorously formulate this PDE, we require some notation. We recall that $[N] = \{1, \dots, N\}$ and denote by $\cal{P}_N$ the power set of $[N]$. For $n \in [N]$, we let $\P_N^n$ consist of all elements of $\P_N$ with cardinality at most $n$. For sets $J \subset I \subset [N]$ and $x \in \R^I$, we introduce the notation $x^J = (x_i)_{i \in J}$. In particular, for $i \in I$, we set $x^{-i} = x^{I \setminus i}$. Here $I \setminus i$ is shorthand for $I \setminus \{i\}$ and we use the convention $\R^{\varnothing} = \{0\}$. Next, for a family $v = (v^{J, j})_{j \in [N], J \in \P_N^n}$ of functions $v^{J, j} \define [0, T] \times \R^J \to \R$, with $n \in \{0, \dots, N\}$, and nonempty $I \in \P_N$ with $\lvert I \rvert \leq (n + 1) \land N$, define the \textit{domains}
\begin{equation} \label{eq:domain}
    \cal{D}^I_Tv = \biggl\{(t, x) \in [0, T] \times \R^I \define x_i > \sum_{j = 1}^N D_{ij} v^{I \setminus i, j}_t(x^{-i}) \text{ for } i \in I\biggr\}
\end{equation}
and $\cal{D}^Iv = \cal{D}^I_Tv \setminus (\{T\} \times \R^I)$, the \textit{index set} of killed particles
\begin{equation} \label{eq:index_set}
    \cal{I}^Iv(t, x) = \biggl\{i \in I \define x_i \leq \sum_{j = 1}^N D_{ij} v^{I \setminus i, j}_t(x^{-i})\Biggr\}
\end{equation}
for $(t, x) \in [0, T] \times \R^I$, as well as the \textit{boundary functions} $\cal{F}^I_iv \define [0, T] \times \R^I \to \R$, $i \in I$, by
\begin{equation} \label{eq:bc}
    \cal{F}^I_iv(t, x) = v^{I \setminus i_0, i}_t(x^{-i_0})
\end{equation}
for $(t, x) \in [0, T] \times \R^I$, where $i_0$ is the minimal index in $\cal{I}^Iv(t, x)$. Note that the family $v = (v^{I, i})_{i \in [N], I \in \P_N^n}$ is only defined for elements in $\P_N^n$ but determines the domains, index set, and boundary functions for all subsets $I$ in $\P_N$ with cardinality at most $(n + 1) \land N$. Note further that the definitions \eqref{eq:domain}, \eqref{eq:index_set}, and \eqref{eq:bc} for this $I$ only draw on $v^{I \setminus i, j}$ for $j \in [N]$. That is, only levels of the PDE below $I$ are required.

Let us remark here that the definition of the domains $\cal{D}^I_T v$ and $\cal{D}^I v$, demarcating states $x \in \R^I$ where all particles are alive, is perhaps surprising at first glance. Indeed, in the probabilistic formulation, FBSDE \eqref{eq:fbsde_hetero}, an alive particle $i \in I$ is killed the first time $t \in [0, T]$ that its state $X^i_t$ crosses below $\sum_{j = 1}^N D_{ij} Y^j_t$. Thus, the anticipated relationship $Y^j_t = v^{I, j}_t(\bf{X}^I_t)$ suggests that the domain should be given by
\begin{equation*}
    \cal{D}^I_Tv = \biggl\{(t, x) \in [0, T] \times \R^I \define x_i > \sum_{j = 1}^N D_{ij} v^{I, j}_t(x^{-i}) \text{ for } i \in I\biggr\}.
\end{equation*}
Note that in the above, the occurrence of $v^{I \setminus i, j}$ in \eqref{eq:domain} is replaced by $v^{I, j}$. Since, as we shall see below, $\cal{D}^I_Tv$ serves as a domain for the PDE satisfied by $v^{I, j}$, $j \in [N]$, this would mean that the domain depends on the solution $v^{I, j}$ itself. What we are exploiting in order to circumvent this circularity, is that at the time $t \in [0, T]$ at which one of the particles $i \in I$ is actually killed, we will have that $Y^j_t = v^{I \setminus i, j}_t(\bf{X}^{I \setminus i}_t)$ for $j \in [N]$. Thus, as long as
\begin{equation*}
    X^i_t > \sum_{j = 1}^N D_{ij} v^{I \setminus i, j}_t(\bf{X}^{I \setminus i}_t),
\end{equation*}
no killing event can have occurred. This motivates the replacement of $v^{I, j}$ by $v^{I \setminus i, j}$ in the definition of the domain in \eqref{eq:domain}.

Let us finally move to the statement of the PDEs for the decoupling field. For a family $v = (v^{I, i})_{i \in [N], I \in \P_N}$ of functions $v^{I, i} \define [0, T] \times \R^I \to \R$, we consider the cascade of PDEs, indexed by $I \in \P_N$:
\begin{equation} \label{eq:pde_moving_boundary}
    \begin{cases}
        \partial_t v^{I, i}_t(x) + \frac{\sigma^2}{2} \Delta v^{I, i}_t(x) = 0 & \text{for } (t, x) \in \cal{D}^I v , \\
        v^{I, i}_t(x) = \cal{F}^I_iv(t, x) & \text{for } (t, x) \in ([0, T] \times \R^I) \setminus \cal{D}^I_Tv, \\
        v^{I, i}_T(x) = 0 &  \text{for } (T, x) \in \cal{D}^I_Tv
    \end{cases}
\end{equation}
for all $i \in I$ and $v^{I, i}_t(x) = 1$ for $(t, x) \in [0, T] \times \R^I$ for all $i \in [N] \setminus I$. Let us explicate PDE \eqref{eq:pde_moving_boundary} in words. If the particle is already killed, i.e.\@ $i \notin I$, its conditional killing probability is one, so $v^{I, i}_t(x) = 1$. Next, if $i \in I$, then in the domain $\cal{D}^I v$, the decoupling field $v^{I, i}$ solves a simple heat equation, corresponding to the fact that the state is driven by a Brownian motion (with volatility $\sigma$). Outside the domain $\cal{D}^I_T v$, at least one of the particles in $I$ is dead, so $v^{I, i}_t(x)$ coincides with the value
\begin{equation*}
    \F^I_i v(t, x) = v^{I \setminus i_0, i}_t(x^{-i_0})
\end{equation*}
of the reduced system $I \setminus i_0$, where $i_0$ is the dead particle with the smallest minimal index. This is a rather arbitrary choice and as we prove in Lemma \ref{lem:mon_well} below, that the value $v^{I \setminus j, i}_t(x^{-j})$ agrees among all indices $j \in I$ of dead particles. Finally, if all particles are alive at the final time $T$, no risk of being killed remains, so that $v^{I, i}_T(x) = 0$.

We are interested in classical solutions of PDE \eqref{eq:pde_moving_boundary}. Let us carefully define this concept.

\begin{definition} \label{def:classical_solution}
We say that a family $v = (v^{I, i})_{i \in [N], I \in \P_N^n}$ of functions $v^{I, i} \define [0, T] \times \R^I \to \R$ is a \textit{classical solution} of the system of PDEs \eqref{eq:pde_moving_boundary} up to level $n \in [N]$ if for all $I \in \P_N^n$, we have
\begin{enumerate}[noitemsep, label = (\roman*)]
    \item \label{it:regularity} $\cal{D}^Iv$ is an open subset of $[0, T) \times \R^I$, $v^{I, i} \in C^{1, 2}(\cal{D}^Iv)$, and $v^{I, i} \in C_b([0, T) \times \R^I)$;
    \item \label{it:pde} $v^{I, i}$ satisfies the equations in \eqref{eq:pde_moving_boundary};
    \item \label{it:terminal} for all compact sets $K \subset \R^I$ such that $\{T\} \times K \subset \cal{D}^I_Tv$,
    \begin{equation*}
        \lim_{t \to T} \sup_{x \in K} \bigl\lvert v^{I, i}_t(x) - v^{I, i}_T(x)\bigr\rvert = 0
    \end{equation*}
\end{enumerate}
if $i \in I$ and $v^{I, i}_t(x) = 1$ for $(t, x) \in [0, T] \times \R^I$ if $i \in [N] \setminus I$.

A \textit{classical solution} of the system of PDEs \eqref{eq:pde_moving_boundary} is a classical solution up to level $N$.
\end{definition}

The differentiability assumption in \ref{it:regularity} is required for heat equation in \eqref{eq:pde_moving_boundary} to be classically well-posed. The continuity of the solution is used to establish uniqueness for PDE \eqref{eq:pde_moving_boundary} and is crucial for the verification result, Theorem \ref{thm:verification}, which connects PDE \eqref{eq:pde_moving_boundary} with FBSDE \eqref{eq:fbsde_hetero}. Lastly, Property \ref{it:terminal} is another ingredient employed in the uniqueness proof for PDE \eqref{eq:pde_moving_boundary}.

Next, let us introduce the concept of a nonincreasing solution. This turns out to be a quite useful property for the analysis of PDE \eqref{eq:pde_moving_boundary}, which is is naturally satisfied by the unique classical solution. From the perspective of the particle system, it states that the further away from the final time one starts the system and the lower the states of the individual particles are, the higher the probability of being killed until the terminal horizon.

\begin{definition}
We say that a classical solution $(v^{I, i})_{i \in [N], I \in \P_N^n}$ of PDE \eqref{eq:pde_moving_boundary} (up to level $n \in [N]$) is \textit{nonincreasing} if for all $I \in \P_N^n$ and $i \in I$, it holds that
\begin{equation*}
    v^{I, i}_t(x) \leq v^{I, i}_s(y)
\end{equation*}
whenever $(t, x)$, $(s, y) \in [0, T] \times \R^I$ with $s \leq t$ and $y \leq x$.
\end{definition}

Our first objective will be to prove the existence and uniqueness of a classical solution to the system of PDEs \eqref{eq:pde_moving_boundary}. Subsequently, we establish a rigorous connection between PDE \eqref{eq:pde_moving_boundary} and FBSDE \eqref{eq:fbsde_hetero} by constructing a solution $(X^i, Y^i, Z^i)_{i \in [N]}$ to the latter from a solution $(v^{I, i})_{i \in [N], I \in \P_N}$ of the former. As hinted to earlier, $v^{I, i}_t(x)$ will provide the conditional killing probability $Y^i_t$ if at time $t$ the set of living particles is $I$ and the state of particle $j \in I$ is given by $X^j_t = x_j$.

Let us conclude these introductory remarks regarding PDE \eqref{eq:pde_moving_boundary} with a numerical illustration of its solution in the symmetric setting $D_{ij} = \alpha$ for $i \in [N]$ and some $\alpha > 0$. From the specification of the PDE, we immediately obtain that $v^{\emptyset, i}_t(0) = 1$. Plugging this into the equation for the first level, we see that the boundary stays fixed in time, which allows us to solve the first level explicitly. Indeed, for $I \in \P^1_N$, we have $v^{I, j}_t(x) = 1$ if $j \notin I$, while for $i \in I$, the function $v^{I, i}$ satisfies
\begin{equation*}
    \partial_t v^{I, i}_t(x) + \frac{\sigma^2}{2} \partial_x^2 v^{I, i}_t(x) = 0
\end{equation*}
for $(t, x) \in [0, T) \times (\alpha N, \infty)$, with the boundary condition $v^{I, i}_t(x) = 1$ for $(t, x) \in [0, T] \times (-\infty, \alpha N]$ and the terminal condition $v^{I, i}_T(x) = 0$ for $x \in (\alpha N, \infty)$. One may verify that the unique solution to this PDE is given by
\begin{equation*}
    v^{I, i}_t(x) = 2\Phi\biggl(\frac{(\alpha N - x) \land 0}{\sqrt{T - t}}\biggr).
\end{equation*}
In particular, the solution does not depend on which singleton $I \in \P^1_N$ we select.

\begin{figure}[tb] 
    \makebox[\linewidth][c]{
    \begin{subfigure}[b]{0.5\columnwidth}
        \centering
        \includegraphics[width=\columnwidth]{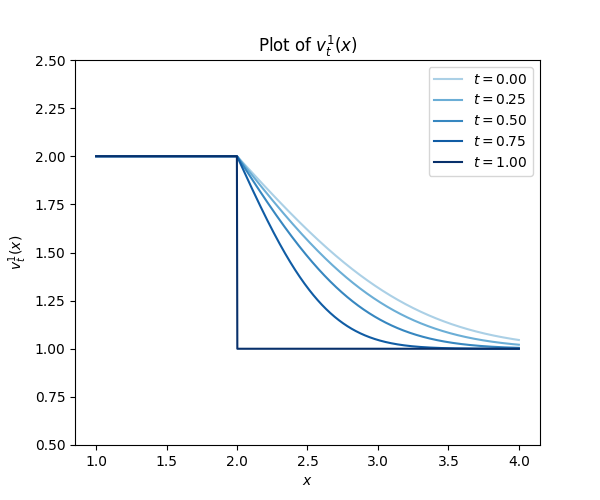}
    \end{subfigure}

    \hspace{-0.5cm}

    \begin{subfigure}[b]{0.555\columnwidth}
        \centering
        \includegraphics[width=\columnwidth]{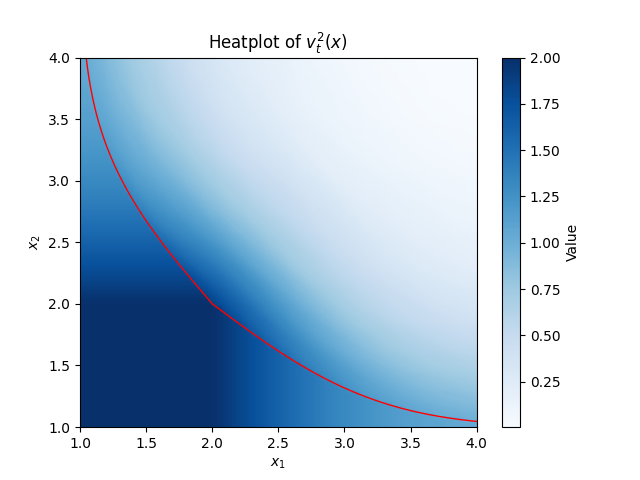}
    \end{subfigure}
    }
    \caption{The plots show the conditional killing probability for a system with two particles. The left-hand side shows the probability if one particle was already removed, while on the right-hand side both particles are initially alive.}
    \label{fig:solution_plot}
 \end{figure}

The higher levels cannot be solved explicitly, so we use a numerical scheme based on the construction of the solution used in the proof of Theorem \ref{thm:pde_exist_unique} below. We simulate a system with two particles and parameters $\sigma = 1$, $\alpha = 1$, and $T = 1$. On the left hand side of Figure \ref{fig:solution_plot}, we plot the total conditional killing probability in the case that one particle has already been killed. That is, the plot shows $x \mapsto v^1_t(x) = v^{I, i}_t(x) + 1$ for arbitrary $i \in I \in \P^1_N$ at different times $t \in [0, 1]$. We can see that if the state $x$ lies below the killing threshold $\alpha N = 2\alpha$, both particle are killed, so the total conditional killing probability equals two. As the state moves further away from the boundary, the probability is reduced. 

The plot on the right-hand side shows a heatmap of total conditional killing probability
\begin{equation*}
    x \mapsto v^2_t(x) = v^{[2], 1}_t(x) + v^{[2], 2}_t(x)
\end{equation*}
for $t = 0$, when none of the particles was initially removed. Dark blue colouration corresponds to a high killing probability, while a low probability is indicated in light blue. The red line demarcates the domain's boundary. Again the probability of being killed falls with the distance of the states to the boundary. Below the boundary, at least one of the particle is dead, while in the dark blue square in the lower left, both particles are removed. Note that the boundary has a kink at $x = (2\alpha, 2\alpha)$. In higher dimensions, the time-regularity of the boundary at these kinks is not straightforwardly verified. If the boundary were to recede too rapidly at such kinks as time unfolds, the solution constructed in the proof of Theorem \ref{thm:pde_exist_unique} could fail to be continuous up to the boundary at the kink. Guaranteeing that such a situation cannot arise is the main challenge in the proof of Theorem \ref{thm:pde_exist_unique}.

\subsection{Existence and Uniqueness for PDE  \texorpdfstring{\eqref{eq:pde_moving_boundary}}{(EQ)}}

We begin by stating the existence and uniqueness result for PDE \eqref{eq:pde_moving_boundary}.

\begin{theorem} \label{thm:pde_exist_unique}
There exists a unique classical solution $(v^{I, i})_{i \in [N], I \in \P_N}$ to the system of PDEs \eqref{eq:pde_moving_boundary}. Moreover, the solution is nonincreasing and for all $I \in \P_N$ and $i \in I$, it holds that $0 \leq v^{I, i} \leq 1$.
\end{theorem}

As a preparation for the proof of Theorem \ref{thm:pde_exist_unique}, we establish a sequence of lemmas, which are grouped by the features of the PDE they address. To avoid repetition, when stating the lemmas, it is assumed that
\begin{equation*}
    v = (v^{I, i})_{i \in [N], I \in \P_N^{n - 1}}
\end{equation*}
is a classical solution of the system of PDEs \eqref{eq:pde_moving_boundary} up to level $n - 1$ for some $n \in [N]$.

\subsubsection{Boundary Condition}

The first lemma shows that the definition of the boundary condition in \eqref{eq:bc} does not depend on the choice of the index $i_0 \in \cal{I}^Iv(t, x)$. In the course of this, we establish the monotonicity between different levels of nonincreasing solutions. 

\begin{lemma} \label{lem:mon_well}
Assume that $v$ is nonincreasing. Then for any nonempty $I \in \P_N^n$ and $(t, x) \in [0, T] \times \R^I$, we have
\begin{enumerate}[noitemsep, label = (\roman*)]
    \item \label{it:well_definedness} $v^{I \setminus \ell, i}_t(x^{-\ell})$ is constant over all $\ell \in \cal{I}^Iv(t, x)$;
    \item \label{it:monotonicity} $v^{J, i}_t(x^J) \geq v^{I, i}_t(x)$ for all $i \in I$ and $J \subset I$.
\end{enumerate}
\end{lemma}

\begin{proof}
It is clearly enough to show that the first result holds in the case that $I = J \cup \{k_0\}$ for some $k_0 \in I$. The general result then follows by successive application of this special case. So from now on, let us assume that $J$ is of this form. We establish both results by a single induction on the cardinality $m \in [n]$ of the set $I$. To start with, let $I$ be the singleton $\{k_0\}$, so that $J = \varnothing$. Then \ref{it:well_definedness} is trivially satisfied. For \ref{it:monotonicity} we note that if $(t, x) \in [0, T] \times \R^I$ lies outside of the domain $\cal{D}^I_T$, then it follows from the boundary condition that
\begin{equation*}
    v^{I, k_0}_t(x) = v^{I \setminus k_0, k_0}_t(x^{-k_0}) = v^{J, k_0}_t(x^J).
\end{equation*}
On the other hand, if $(t, x)$ is inside the domain, we have by nonincreasingness of the solution that
\begin{equation*}
    v^{I, k_0}_t(x) \leq v^{I, k_0}_t(x_0) = v^{I \setminus k_0, k_0}_t(x_0^{-k_0}) = v^{I \setminus k_0, k_0}_t(x^{-k_0}) = v^{J, k_0}_t(x^J),
\end{equation*}
where $x_0 \leq x$ is some point in $\R^I$ such that $(t, x_0)$ lies outside the domain. Here we used in the second equality that $x_0^{-k_0} = 0 = x^{-k_0}$ since $I$ is a singleton. This concludes the induction start. 

Now, suppose $\lvert I \rvert = m \in \{2, \dots, n\}$ and that the result holds for all sets with cardinality less than or equal to $m - 1$. We establish \ref{it:monotonicity} first. Consider the case that $(t, x) \in [0, T] \times \R^I$ such that $k_0 \in \cal{I}^Iv(t, x)$. Let $i_0$ be the minimal index in $\cal{I}^Iv(t, x)$, so that due to the boundary condition of PDE \eqref{eq:pde_moving_boundary}, we have
\begin{equation*}
    v^{I, i}_t(x) = v^{I \setminus i_0, i}_t(x^{-i_0})
\end{equation*}
for all $i \in I$. By the induction hypothesis for \ref{it:monotonicity}, setting $\tilde{I} = I \setminus i_0$ and $\tilde{x} = x^{-i_0}$, it holds that $v^{\tilde{I}, i}_t(\tilde{x}) \leq v^{\tilde{I} \setminus k_0, i}_t(\tilde{x}^{-k_0})$ for $i \in \tilde{I}$. Furthermore, $v^{\tilde{I}, i}_t(\tilde{x}) = 1 = v^{\tilde{I} \setminus k_0, i}_t(\tilde{x}^{-k_0})$ for $i \in [N] \setminus \tilde{I}$. Since $k_0 \in \cal{I}^Iv(t, x)$ by assumption, this implies
\begin{equation*}
    x_{k_0} \leq \sum_{j = 1}^N D_{k_0j} v^{I \setminus k_0, j}(x^{-k_0}) \leq \sum_{j = 1}^N D_{k_0j} v^{\tilde{I} \setminus k_0, j}(\tilde{x}^{-k_0}),
\end{equation*}
so the particle $k_0$ is dead in the system $\tilde{I}$, by which we mean $k_0 \in \cal{I}^{\tilde{I}}v(t, \tilde{x})$. Thus, it follows from the induction hypothesis for \ref{it:well_definedness} and the boundary condition of PDE \eqref{eq:pde_moving_boundary} that
\begin{equation*}
    v^{\tilde{I}, i}_t(\tilde{x}) = v^{\tilde{I}\setminus k_0, i}_t(\tilde{x}^{-k_0}).
\end{equation*}
From this we can deduce that
\begin{equation*}
    x_{i_0} \leq \sum_{j = 1}^N D_{i_0j} v^{\tilde{I}, j}_t(\tilde{x}) = \sum_{j = 1}^N D_{i_0j} v^{\tilde{I}\setminus k_0, j}_t(\tilde{x}^{-k_0}),
\end{equation*}
meaning that $i_0$ is dead in the system $I \setminus k_0$. Hence, applying the induction hypothesis for \ref{it:well_definedness} and the boundary condition of PDE \eqref{eq:pde_moving_boundary} once more, we find that
\begin{equation*}
    v^{J, i}_t(x^J) = v^{I \setminus k_0, i}_t(x^{-k_0}) = v^{\tilde{I} \setminus k_0, i}_t(\tilde{x}^{-k_0}) = v^{I \setminus i_0, i}_t(x^{-i_0}) = v^{I, i}_t(x)
\end{equation*}
for $i \in I$, giving \ref{it:monotonicity}. Note, moreover, that since the element $k_0 \in \cal{I}^Iv(t, x)$ can be chosen arbitrarily, this also proves \ref{it:well_definedness}. Thus it only remains to conclude the induction step for \ref{it:monotonicity} in the case that $k_0 \notin \cal{I}^Iv(t, x)$. If $k_0$ is not dead in $I$, then by choosing $x_0 \leq x_{k_0}$ small enough such that $k_0 \in \cal{I}^Iv(t, y)$, where $y_i = x_i$ if $i \neq k_0$ and $y_{k_0} = x_0$, we get from the nonincreasingness of the solution that
\begin{equation*}
    v^{I, i}_t(x) \leq v^{I, i}_t(y) = v^{I \setminus k_0, i}_t(y^{-k_0}) = v^{I \setminus k_0, i}_t(x^{-k_0}) = v^{J, i}_t(x^J),
\end{equation*}
where we use that $y^{-k_0} = x^{-k_0}$, since $y$ and $x$ only differ in the component $k_0$. This concludes the induction and the proof.
\end{proof}

From Lemma \ref{lem:mon_well} \ref{it:well_definedness}, it follows that the boundary condition $\F^I_iv$ is continuous.

\begin{lemma} \label{lem:boundary_cond_cont}
Assume that $v$ is nonincreasing. Then, for all nonempty $I \in \P_N^n$ and $i \in I$, the function $\F^I_iv$ is continuous on $([0, T) \times \R^I) \setminus \cal{D}^Iv$.
\end{lemma}

\begin{proof} 
From Lemma \ref{lem:mon_well} \ref{it:well_definedness}, we know that 
\begin{equation} \label{eq:bc_all}
    \F^I_iv(t, x) = v^{I \setminus k, i}_t(x^{-k})
\end{equation}
for all $k \in \cal{I}^Iv(t, x)$. Now, suppose that $(t_m, x^m)_{m \geq 1}$ is a sequence in $([0, T) \times \R^I) \setminus \cal{D}^Iv$ which converges to $(t, x) \in ([0, T) \times \R^I) \setminus \cal{D}^Iv$. Let $i_m$ be the minimal element of $\cal{I}^Iv(t_m, x^m)$. We will show that any convergent subsequence of
\begin{equation*}
    \F^I_iv(t_m, x^m) = v^{I \setminus i_m, i}_{t_m}\bigl((x^m)^{-i_m}\bigr)
\end{equation*}
has the same limit $\F^I_iv(t, x)$. Since the latter is independent of the convergent subsequence, the desired continuity statement
\begin{equation*}
    \lim_{m \to \infty} \F^I_iv(t_m, x^m) = \F^I_iv(t, x)
\end{equation*}
follows. So, fix a convergence subsequence with corresponding subsequence of indices $(m_{\ell})_{\ell \geq 1}$. By selecting a further subsequences if necessary, we may assume that the indices $j_{\ell} = i_{m_{\ell}}$ converge to some $i_0 \in I$. In particular, the sequence $(j_{\ell})_{\ell \geq 1}$ becomes stationary for $\ell$ large enough. Hence, using the continuity of $v^{I \setminus i_0, j}$ on $[0, T) \times \R^{I_0}$, we find that
\begin{equation*}
    x_{i_0} = \lim_{\ell \to \infty} y^{\ell}_{j_{\ell}} \leq \lim_{\ell \to \infty} \sum_{j = 1}^N D_{j_{\ell}j} v^{I \setminus j_{\ell}, j}_{s_{\ell}}\bigl((y^{\ell})^{-j_{\ell}}\bigr) = \sum_{j = 1}^N D_{i_0j} v^{I \setminus i_0, j}_t(x^{-i_0}),
\end{equation*}
where $s_{\ell} = t_{m_{\ell}}$ and $y^{\ell} = x^{j_{\ell}}$. This means that $i_0 \in \cal{I}^Iv(t, x)$, so by \eqref{eq:bc_all}, we have
\begin{align*}
    \lim_{m \to \infty} \F^I_iv(t_m, x^m) &= \lim_{\ell \to \infty} \F^I_iv(s_{\ell}, y^{\ell}) \\
    &= \lim_{\ell \to \infty} v^{I \setminus j_{\ell}, i}_{s_{\ell}}\bigl((y^{\ell})^{-j_{\ell}}\bigr) \\
    &= v^{I \setminus i_0, i}_t(x^{-i_0}) \\
    &= \F^I_iv(t, x).
\end{align*}
This concludes the proof.
\end{proof}

Depending on the location of a state $x \in \R^I$ outside the domain $\cal{D}^I_T v$, more than one particle in $I$ can be dead. To address this situation, the following lemma provides a necessary and sufficient condition for a subset of particles in $I$ to be killed, which only draws on the decoupling field for the configuration $J \subset I$ of the remaining particles. 

\begin{lemma} \label{lem:mult_default}
Assume that $v$ is nonincreasing. Let $I \in \P_N^n$ be nonempty, $(t, x) \in [0, T] \times \R^I$, and $J \subset I$. Then $x_i \leq \sum_{j = 1}^N D_{ij} v^{J, j}_t(x^J)$ for all $i \in I \setminus J$ if and only if $I \setminus J \subset \cal{I}^Iv(t, x)$. In particular,
\begin{equation} \label{eq:bc_down}
    v^{J, i}_t(x) = v^{I, i}_t(x)
\end{equation}
for all $J \subset I$ with $I \setminus J \subset \cal{I}^I v(t, x)$ and $i \in I$.
\end{lemma}

\begin{proof}
Fix $I \in \P_N^n$ and $(t, x) \in [0, T] \times \R^I$. Let us begin with the more involved only if statement. We establish the result by induction on the cardinality $m \in [\lvert I\rvert]$ of the set $I \setminus J$. The case $\lvert I \setminus J \rvert = 1$ is trivial. Next, let us assume the result holds for all subsets $J$ of $I$ such that $\lvert I \setminus J \rvert \leq m - 1$ for some $m \in \{2, \dots, \lvert I\rvert\}$. Our goal is to show it also holds for subsets whose complement has cardinality $m$. Let $J \subset I$ have this property and fix $i_0 \in I \setminus J$. Since $I \setminus J$ has at least two elements we can select another index $j_0 \in I \setminus J$, so that
\begin{equation*}
    x_{j_0} \leq \sum_{j = 1}^N D_{j_0j} v^{J, j}_t(x^J).
\end{equation*}
This means that the particle with index $j_0$ is killed when viewed as a member of the system with particles $\tilde{J} = J \cup \{j_0\}$. In particular, owing to Lemma \ref{lem:mon_well} \ref{it:well_definedness}, we have $v^{\tilde{J}, j}_t(x^{\tilde{J}}) = v^{J, j}_t(x^J)$ for $j \in \tilde{J}$, while $v^{\tilde{J}, j}_t(x^{\tilde{J}}) = 1 = v^{J, j}_t(x^J)$ for $j \in [N] \setminus \tilde{J}$. Consequently, for any $i \in I \setminus \tilde{J}$, it holds that
\begin{equation*}
    x_i \leq \sum_{j = 1}^N D_{ij} v^{J, j}_t(x^J) = \sum_{j = 1}^N D_{ij} v^{\tilde{J}, j}_t(x^{\tilde{J}}).
\end{equation*}
But then the set $\tilde{J}$ which contains $i_0$ and whose complement $I \setminus \tilde{J}$ has cardinality $m - 1$ satisfies the same property as $J$ in the statement of the lemma. Thus, by the induction hypothesis, it holds that $i_0 \in \cal{I}^Iv(t, x)$. Since $i_0$ was an arbitrary element of $I \setminus J$, it follows that $I \setminus J \subset \cal{I}^Iv(t, x)$. This concludes the induction.

To establish the if statement, let us fix $J \subset I$ with $I \setminus J \subset \cal{I}^Iv(t, x)$. Note that if $J = \varnothing$, the statement is trivial, so we will assume that $J$ has at least one element. Fix such an index $i_1 \in I \setminus J$ and choose an arbitrary enumeration $i_2$,~\ldots, $i_m$ of the remaining elements of $I \setminus J$. Now, by induction on $k = m$,~\ldots, $1$, we will show that $i_k \in \cal{I}^{J_k}v(t, x^{J_k})$ and $v^{J_{k - 1}, j}_t(x^{J_{k - 1}}) = v^{I, j}_t(x)$ for $j \in I$, where $J_k = J \cup \{i_1, \dots, i_k\}$. For $k = m$, the induction statement is clear since $i_m \in \cal{I}^Iv(t, x)$ by assumption. Next, suppose the result holds for $k + 1$ with $k \in \{1, \dots, m - 1\}$. Then by the induction hypothesis and Lemma \ref{lem:mon_well} \ref{it:monotonicity}, we have
\begin{equation*}
    x_{i_k} \leq \sum_{j = 1}^N D_{i_k j} v^{I, j}_t(x) = \sum_{j = 1}^N D_{i_k j} v^{J_k, j}_t\bigl(x^{J_k}\bigr) \leq \sum_{j = 1}^N D_{i_k j} v^{J_{k - 1}, j}_t\bigl(x^{J_{k - 1}}\bigr),
\end{equation*}
so that $i_k \in \cal{I}^{J_k}v(t, x^{J_k})$. Thus, from Lemma \ref{lem:mon_well} \ref{it:well_definedness} and another application of the induction hypothesis, we obtain
\begin{equation*}
    v^{J_{k - 1}, j}_t(x^{J_{k - 1}}) = v^{J_k, j}_t(x^{J_k}) = v^{I, j}_t(x) 
\end{equation*}
for all $j \in I$. This terminates the induction. For $k = 1$, the induction statement reads 
\begin{equation*}
    x_{i_1} \leq \sum_{j = 1}^N D_{i_1j} v^{J, j}_t(x^J)
\end{equation*}
and $v^{J, j}_t(x^J) = v^{I, j}_t(x)$ for all $j \in I$. Since $i_1$ was an arbitrary element of $I \setminus J$, this proves both the if statement and \eqref{eq:bc_down}.
\end{proof}

The next lemma states that after removing all particles in $\cal{I}^I(t, x)$ from $I$, i.e.\@ the ones which were killed, we obtain a system in which all particles are alive. Moreover, $I \setminus \cal{I}^I(t, x)$ is the maximal subsystem of $I$ with that property.

\begin{lemma} \label{lem:mult_default_down}
Assume that $v$ is nonincreasing. Let $I \in \P_N$ be nonempty and $(t, x) \in [0, T] \times \R^I$. Then $I \setminus \cal{I}^Iv(t, x)$ is the union over all sets $J \subset I$ such that $(t, x^J) \in \cal{D}^J_T v$.
\end{lemma}

\begin{proof}
Set $J_1 = I \setminus \cal{I}^Iv(t, x)$ and let $J_2$ denote the union in the statement of the lemma. We shall show $J_1 \subset J_2$ and $J_2 \subset J_1$ separately. Let us begin with the former. Clearly, it is enough to prove that $(t, x^{J_1}) \in \cal{D}^{J_1}_Tv$. Suppose this is not the case and let $i_0 \in \cal{I}^{J_1}v(t, x^{J_1})$, so that
\begin{equation*}
    x_{i_0} \leq \sum_{j = 1}^N D_{i_0j} v^{J_1 \setminus i_0, j}_t(x^{J_1 \setminus i_0}).
\end{equation*}
We furthermore have by Lemma \ref{lem:mon_well} \ref{it:monotonicity} that
\begin{equation*}
    x_i \leq \sum_{j = 1}^N D_{ij} v^{J_1, j}_t(x^{J_1}) \leq \sum_{j = 1}^N D_{ij} v^{J_1 \setminus i_0, j}_t(x^{J_1 \setminus i_0})
\end{equation*}
for all $i \in \cal{I}^Iv(t, x)$, where we applied the if direction of Lemma \ref{lem:mult_default} in the first inequality. But this means $x_i \leq \sum_{j = 1}^N D_{ij} v^{J_1 \setminus i_0, j}_t(x^{J_1 \setminus i_0})$ for all $i \in I \setminus (J_1 \setminus i_0) = \cal{I}^Iv(t, x) \cup \{i_0\}$. Thus, employing the only if statement of Lemma \ref{lem:mult_default}, we find that $\cal{I}^Iv(t, x) \cup \{i_0\} = I \setminus (J_1 \setminus i_0) \subset \cal{I}^Iv(t, x)$. But $i_0 \in J_1$ and $J_1 \cap \cal{I}^Iv(t, x) = \varnothing$, so the above inclusion is a contradiction. Thus, $(t, x^{J_1}) \in \cal{D}^{J_1}_Tv$ as required.

To prove $J_2 \subset J_1$, it suffices to show that for all $J \subset I$ with $(t, x^J) \in \cal{D}^J_Tv$, we have $J \subset J_1$ or, equivalently, $J \cap \cal{I}^Iv(t, x) = \varnothing$. Suppose the latter is not the case and let $i$ lie in the intersection of $J$ and $\cal{I}^Iv(t, x)$. Then by Lemma \ref{lem:mon_well} \ref{it:monotonicity}, we have
\begin{equation*}
    x_i \leq \sum_{j = 1}^N D_{ij} v^{I \setminus i, j}_t(x^{-i}) \leq \sum_{j = 1}^N D_{ij} v^{J \setminus i, j}_t(x^{J \setminus i}),
\end{equation*}
which yields $i \in \cal{I}^Jv(t, x^J)$. This, however, contradicts $(t, x^J) \subset \cal{D}^J_Tv$.
\end{proof}

\subsubsection{Terminal Value}

The terminal value of a solution to PDE \eqref{eq:pde_moving_boundary} can be explicitly computed from the boundary and terminal condition.

\begin{lemma} \label{lem:terminal_condition}
Assume that $v$ is nonincreasing. Then, for all $I \in \P_N^{n - 1}$, $i \in I$, and $x \in \R^I$, we have
\begin{equation*}
    v^{I, i}_T(x) = \bf{1}_{i \in \cal{I}^Iv(T, x)}.
\end{equation*}
\end{lemma}

\begin{proof}
By Lemma \ref{lem:mult_default}, we have that $v^{I, i}_T(x) = v^{J, i}_T(x)$, where $J = I \setminus \cal{I}^Iv(t, x)$. Lemma \ref{lem:mult_default_down} implies that $(T, x^J) \in \cal{D}_T^Jv$, so that $v^{I, i}_T(x) = v^{J, i}_T(x) = 0$ for $i \in J = I \setminus \cal{I}^Iv(t, x)$ by the terminal condition of PDE \eqref{eq:pde_moving_boundary}. For $i \in \cal{I}^Iv(t, x)$, it trivially holds that $v^{I, i}_T(x) = v^{J, i}_T(x) = 1$. The desired expression for $v^{I, i}_T(x)$ follows.
\end{proof}

By a simple induction, we can show that a solution of PDE \eqref{eq:pde_moving_boundary} approaches the terminal value in a continuous manner.

\begin{lemma} \label{eq:cont_to_terminal}
Assume that $v$ is nonincreasing. Then, for all $I \in \P_N^{n - 1}$, $i \in I$, and $x \in \R^I$, we have
\begin{equation*}
    \lim_{t \nearrow T} v^{I, i}_t(x) = v^{I, i}_T(x)
\end{equation*}
\end{lemma}

\begin{proof}
We prove the lemma by induction on the cardinality $m \in \{0, \dots, n - 1\}$ of $I$. For $m = 0$, the result is trivial since $v^{I, i}_t(x) = 1$ for all $t \in [0, T] \times \R^I$. Next, suppose that the result holds for cardinalities $m - 1$ with $m  \in [n]$. If $(T, x) \in \cal{D}^I_Tv$, then the continuity statement follows from the boundary condition of PDE \eqref{eq:pde_moving_boundary}, so let $(T, x)$ be outside the domain $\cal{D}^I_Tv$ instead. Then, due to the assumed nonincreasingness of $v$, we have that $(t, x)$ lies outside of $\cal{D}^I_Tv$. Moreover, we can find $i_0 \in I$ which is a member of $\cal{I}^Iv(t, x)$ for all $t \in [0, T]$. By Lemma \ref{lem:mon_well} \ref{it:well_definedness}, this implies that $v^{I, i}_t(x) = v^{I \setminus i_0, i}_t(x)$ for $t \in [0, T]$, so from the induction hypothesis, we deduce that
\begin{equation*}
    \lim_{t \nearrow T} v^{I, i}_t(x) = \lim_{t \nearrow T} v^{I, i}_t(x) v^{I \setminus i_0, i}_t(x) = v^{I, i}_T(x) v^{I \setminus i_0, i}_t(x) = v^{I, i}_T(x).
\end{equation*}
\end{proof}

\subsubsection{Domain and Boundary}

We derive some elementary regularity properties of the domains $\cal{D}^I_Tv$ and $\cal{D}^Iv$ and their boundaries. We begin by showing that both domains are open.

\begin{lemma} \label{lem:domain}
The domain $\cal{D}^Iv$ is an open subset of $[0, T] \times \R^I$ for all nonempty $I \in \P_N^n$. If $v$ is nonincreasing, the same is true for $\cal{D}^I_Tv$.
\end{lemma}

\begin{proof}
Let $t \in [0, T)$. Then $\cal{D}^I_Tv \cap ([0, t) \times \R^I)$ is open in $[0, T] \times \R^I$
as the preimage of the open set $(0, \infty)^I$ under the continuous map
\begin{equation*}
    (s, x) \mapsto \biggl(x_i - \sum_{j = 1}^N D_{ij} v^{I \setminus i, j}_s(x^{-i})\biggr)_{i \in I}. 
\end{equation*}
Since $\cal{D}^Iv$ is the union of $\cal{D}^I_Tv \cap ([0, t) \times \R^I)$ over $t \in [0, T)$ it is an open subset of $[0, T] \times \R^I$ as well. 

Now, to prove that $\cal{D}^I_Tv$ is open, it is enough to show that for each $(t, x) \in \cal{D}^I_Tv$ there exists $\epsilon > 0$ such that the ball of radius $\epsilon$ in $[0, T] \times \R^I$ centred at $(t, x)$ lies in $(t, x) \in \cal{D}^I_Tv$. For $(t, x) \in \cal{D}^I_Tv$ with $t \in [0, T)$, this is clear by openness of $\cal{D}^Iv$ in $[0, T] \times \R^I$. So let us suppose that $x \in \R^I$ with $(T, x) \in \cal{D}^I_Tv$. We claim that for $t \in [0, T)$ sufficiently close to $T$, it holds that $(t, x) \in \cal{D}^I_Tv$. Otherwise, we can find a sequence $(t_m)_{m \geq 1}$ of elements in $[0, T)$ and an index $i \in I$ such that $i \in \cal{I}^Iv(t_m, x)$. From this and Lemma \ref{eq:cont_to_terminal}, we deduce that
\begin{equation*}
    x_i \leq \lim_{m \to \infty} \sum_{j = 1}^N D_{ij} v^{I \setminus i, j}_{t_m}(x^{-i}) = \sum_{j = 1}^N D_{ij} v^{I \setminus i, j}_T(x^{-i}),
\end{equation*}
which would imply that $(T, x)$ lies outside of $\cal{D}^I_Tv$ in contradiction with out assumption on $x$. Hence, we can find $t \in [0, T)$ with $(t, x) \in \cal{D}^I_Tv$. Now, choose $\epsilon > 0$ small enough such that the ball of radius $\epsilon$ in $[0, T] \times \R^I$ around $(t, x)$ is included in $\cal{D}^I_Tv$. Then nonincreasingness of the solution implies that the ball with the same radius $\epsilon$ but centred at $(T, x)$ instead of $(t, x)$ must be a subset of $\cal{D}^I_Tv$ as well. 
\end{proof}

By applying a transformation to the boundary of $\cal{D}^I_T v$ and appealing to the assumed nonincreasingness of $v$, we can deduce that it is $1$-Lipschitz in space.

\begin{lemma} \label{lem:boundary}
Assume that $v$ is nonincreasing. Then, the boundary of the set $\{x \in \R^I \define (t, x) \in \cal{D}^I_Tv\}$ is $1$-Lipschitz for all $t \in [0, T]$ and all nonempty $I \in \P_N^n$. In particular, it has vanishing Lebesgue measure. 
\end{lemma}

\begin{proof}
Fix $t \in [0, T]$ and a nonempty $I \in \P_N^n$. Let $\bf{1}$ denote the vector in $\R^I$ whose entries are all equal to one. We will establish a $1$-Lipschitz continuous one-to-one relationship between the elements of the plane $E_I = \{y \in \R^I \define \bf{1}^{\top} y = 0\}$ and the boundary of $\{x \in \R^I \define (t, x) \in \cal{D}^I_Tv\}$. To that end, we define the function $r \define E_I \to \R^I$ by
\begin{equation*}
    r(y) = \inf\Bigl\{s \in \R \define (t, y + s\bf{1}) \in \cal{D}^I_Tv\Bigr\}
\end{equation*}
for $y \in E_I$. Note that since $v^{I \setminus j, k}$ is bounded for $j \in I$, $k \in [N]$, the set on the right-hand side above is nonempty and bounded from below. Hence, the infimum is a well-defined value in $\R$. Moreover, since $\cal{D}^I_Tv$ is open in $[0, T] \times \R^I$ by Lemma \ref{lem:domain}, the element $y + r(y) \bf{1}$ lies on the boundary of $\{x \in \R^I \define (t, x) \in \cal{D}^I_Tv\}$ for all $y \in E_I$. 

We will now show that $\{y + r(y) \define y \in E_I\}$ is in fact identical with the boundary of $\{x \in \R^I \define (t, x) \in \cal{D}^I_Tv\}$ and that $r$ and, therefore, also the map $E_I \ni y \mapsto y + r(y) \bf{1}$ are $1$-Lipschitz continuous with respect to the maximum norm $\lvert \cdot \rvert_{\infty}$ on $\R^I$. Since we already argued that $y + r(y) \bf{1}$ lies on the boundary of $\{x \in \R^I \define (t, x) \in \cal{D}^I_Tv\}$, we only need to prove that the latter is included in $\{y + r(y) \define y \in E_I\}$. So let $z$ be on the boundary of $\{x \in \R^I \define (t, x) \in \cal{D}^I_Tv\}$. Then we can find $y \in E_I$ and $s \in \R$ such that $z = y + s\bf{1}$. Thus, we only need to show that $s = r(y)$. If $s$ were larger than $r(y)$, then it follows from the definition of $r(y)$ and the nonincreasingness of the solution that $z = y + s\bf{1}$ lies inside $\cal{D}^I_Tv$. But this is in contradiction with the fact that $z$ is on the boundary of $\{x \in \R^I \define (t, x) \in \cal{D}^I_Tv\}$ and that $\cal{D}^I_Tv$ is open. Hence, we know that $s \leq r(y)$. Suppose now that $s < r(y)$. Then from an analogous argument, it would follow that $y + r(y)$ is inside $\cal{D}^I_Tv$, in contradiction with the definition of the function $r$. Thus, we conclude that $s = r(y)$, as desired.

Finally, let us establish the $1$-Lipschitz continuity of $r$. Fix $y^1$, $y^2 \in E_I$ and set $r_i = r(y^i)$ for $i = 1$, $2$. We must show that $\lvert r_1 - r_2\rvert \leq \lvert y^1 - y^2\rvert_{\infty}$. Suppose this is not the case and that, say, $r_2 > r_1 + \lvert y^1 - y^2\rvert_{\infty}$. Then $y^2 + r_2 \bf{1} > y^1 + r_1 \bf{1}$, since
\begin{equation*}
    y^2_i + r_2 > y^2_i + \lvert y^1 - y^2\rvert_{\infty} + r_1 \geq y^2_i + y^1_i - y^2_i + r_1 = y^1_i + r_1.
\end{equation*}
Now, let us take $s > r_1$ such that $y^2 + r_2\bf{1} \geq y^1 + s \bf{1}$. By definition of $r_1 = r(y^1)$, we have that $y^1 + s \bf{1}$ lies inside $\cal{D}^I_Tv$. But then we may once more appeal to the nonincreasingness of the solution to obtain that
\begin{equation*}
    y^2_i + r_2 \geq y^1_i + s > \sum_{j = 1}^N D_{ij} v^{I\setminus i, j}_t\bigl((y^1 + s\bf{1})^{-i}\bigr) \geq \sum_{j = 1}^N D_{ij} v^{I\setminus i, j}_t\bigl((y^2 + r_2\bf{1})^{-i}\bigr)
\end{equation*}
for all $i \in I$. Consequently, $y^2 + r_2$ is an element of $\cal{D}^I_Tv$, contradicting the definition of $r_2 = r(y^2)$. This concludes the proof.
\end{proof}

\subsection{Proof of Theorem \ref{thm:pde_exist_unique}}

We now have all the prerequisites to prove Theorem \ref{thm:pde_exist_unique}. By using the Feynman--Kac formula for parabolic PDEs in general domains, we can construct a relatively explicit candidate solution (see Equation \eqref{eq:sol_representation}). As indicated earlier, the major challenge is to verify that this candidate is continuous. This challenge originates from the difficulty in establishing the time-regularity of the boundary, which is given by the solution from the previous level. Specifically, the issue that could occur is that at some boundary point $(t, x) \in ([0, T) \times \R^I) \setminus \cal{D}^Iv$, the boundary moves away so rapidly as time unfolds that regardless of how close to the boundary point $x$ we start a Brownian motion, its first hitting time on the boundary will be strictly bounded away from $t$. As a consequence, the solution would not be continuous at $(t, x)$. We circumvent this issue by an application of Lemma \ref{lem:mult_default} above, avoiding the need for technical regularity results for the boundary. Uniqueness for PDE \eqref{eq:pde_moving_boundary} follows from routine arguments.

\begin{proof}[Proof of Theorem \ref{thm:pde_exist_unique}]
\textit{Existence:} We construct a solution $v = (v^{I, i})_{i \in [N], I \in \P_N}$ iteratively. For $\varnothing \in \P_N$ there is nothing to do. Now, assume that $v^{I, i}$ is constructed for all $i \in [N]$ and $I \in \P_N^{n - 1}$ and some $n \in \{0, \dots, N - 1\}$, with the additional properties that $v^{I, i}$ is nonincreasing and $0 \leq v^{I, i} \leq 1$. Set
\begin{equation*}
    v^{(n - 1)} = (v^{I, i})_{i \in [N], I \in \P_N^{n - 1}}
\end{equation*}
and fix $I \in \P_N$ with $\lvert I \rvert = n$. For $i \in I \setminus [N]$, we simply set $v^{I, i}_t(x) = 1$ for all $(t, x) \in [0, T] \times \R^I$. For $i \in I$, we define a candidate solution $v^{I, i}$ as follows: for $(t, x) \in [0, T] \times \R^I$, we let $X^{t, x}_s = x + \sigma (\bf{W}^I_s - \bf{W}^I_t)$ for $s \in [t, T]$ and $\tau_{t, x} = \inf\{s \in [t, T] \define (s, X^{t, x}_s) \notin \cal{D}^I_Tv^{(n - 1)}\}$, where $\inf \varnothing = \infty$. Then, we define $v^{I, i} \define [0, T] \times \R^I \to \R$ by
\begin{equation} \label{eq:sol_representation}
    v^{I, i}_t(x) = \ev\Bigl[\F^I_iv^{(n - 1)}\bigl(\tau_{t, x}, X^{t, x}_{\tau_{t, x}}\bigr) \bf{1}_{\{\tau_{t, x} \leq T\}}\Bigr]
\end{equation}
for $t \in [0, T]$. By Lemma \ref{lem:domain}, the domain $\cal{D}^Iv$ is open and, by construction, $v^{I, i} \in C^{1, 2}(\cal{D}^Iv^{(n - 1)})$ and satisfies the heat equation inside $\cal{D}^Iv^{(n - 1)}$. Next, let $K$ be a compact subset of $\R^I$ with $\{T\} \times K \subset \cal{D}^I_Tv^{(n - 1)}v$. Then $\sup_{x \in K} \pr(\tau_{t, x} \leq T)$ tends to zero as $t \to T$, whence
\begin{equation*}
    \sup_{x \in K} \lvert v^{I, i}_t(x) - v^{I, i}_T(x)\rvert \leq \sup_{x \in K} \pr(\tau_{t, x} \leq T) \to 0
\end{equation*}
as $t \to T$. Thus, we have established the first two items of Property \ref{it:regularity} as well as Properties \ref{it:pde} and \ref{it:terminal}. Since $0 \leq v^{J, j} \leq 1$ and $v^{J, j}$ is nonincreasing for all $j \in [N]$ and $J \in \P_N^{n - 1}$, the same holds true for $v^{I, i}$ by \eqref{eq:sol_representation}. Consequently, it only remains to show that $v^{I, i} \in C_b([0, T) \times \R^I)$. 

Since $v^{I, i} \in C^{1, 2}(\cal{D}^Iv^{(n - 1)})$, continuity of $v^{I, i}$ is clear inside $\cal{D}^Iv^{(n - 1)}$. Next, on $([0, T) \times \R^I) \setminus \cal{D}^Iv^{(n - 1)}$, continuity follows from the boundary condition and Lemma \ref{lem:boundary_cond_cont}. The latter is applicable since the partial solution $v^{(n - 1)}$, is nonincreasing. Hence, it only remains to show that for any convergent sequence $(t_m, x^m)_{m \geq 1}$ in $\cal{D}^Iv^{(n - 1)}$ tending to an element $(t, x) \in [0, T) \times \R^I$ on the boundary of $\cal{D}^Iv^{(n - 1)}$, we have $v^{I, i}_{t_m}(x^m) \to v^{I, i}_t(x)$ as $m \to \infty$. Define $X^{(m)} = X^{t_m , x^m}$, $X = X^{t, x}$, $\tau_m = \tau_{t_m, x^m}$, and $\tau = \tau_{t, x} = t$. Here $\tau_{t, x} = t$ follows from the fact that $(t, x)$ is on the boundary of $\cal{D}^Iv^{(n - 1)}$. We will show that $\tau_m$ converges to $\tau = t$ almost surely. First, we prove that $\limsup_{m \to \infty} \tau_m \leq t$. 

Let $J = I \setminus \cal{I}^Iv^{(n - 1)}(t, x)$, so that $(t, x^J) \in \cal{D}^J_Tv^{(n - 1)}$ by Lemma \ref{lem:mult_default_down}. Consequently, $v^{J, k}$ is $C^{1, 2}$ near $(t, x^J) = (t, X^J_t)$ for all $k \in J$. We will exploit this regularity to show that for any $\delta \in (0, T - t)$ small enough, there exists $s \in (t, t + \delta]$ such that
\begin{equation} \label{eq:mid_step}
    X^k_s < \sum_{j = 1}^N D_{kj} v^{J, j}_s(X^J_s)
\end{equation}
for any $k \in I \setminus J = \cal{I}^Iv^{(n - 1)}(t, x)$. From this, the convergence of $(t_m, X^{(m)}_s)$ to $(t, X_s)$, and the continuity of $(u, y) \mapsto v^{I \setminus k, j}_u(y^{-k})$ on $[0, T) \times \R^I$, we then deduce that for all $m \geq 1$ large enough, it holds that
\begin{equation*}
    X^{(m), k}_s \leq \sum_{j = 1}^N D_{kj} v^{J, j}_s\bigl((X^{(m)}_s)^J\bigr)
\end{equation*}
for all $k \in I \setminus J$. In view of the only if statement of Lemma \ref{lem:mult_default}, we deduce that $\cal{I}^Iv^{(n - 1)}(s, X^{(m)}_s) \supset I \setminus J$ is nonempty, so $(s, X^{(m)}_s)$ lies outside the domain $\cal{D}^I_T v^{(n - 1)}$. This in turn means that $\tau_m \leq s \leq t + \delta$. Letting $\delta \to 0$ implies that $\limsup_{m \to \infty} \tau_m \leq t$. Let us establish \eqref{eq:mid_step}. Note that
\begin{equation*}
    \liminf_{s \searrow t} \Bigl(\max_{j \in I} \frac{W^j_s - W^j_t}{s - t}\Bigr) = -\infty
\end{equation*}
almost surely, so, on some set of full measure that we shall fix from now on, for any $\delta \in (0, T - t)$, we can find $s \in (t, t + \delta]$ such that $\max_{j \in I}(W^j_s - W^j_t) \leq - \frac{s - t}{\delta}$. Now, let us fix a $\delta \in (0, T - t)$ and the corresponding $s \in (t, t + \delta]$, so that in particular $X^k_s < X^k_t = x_k$ for all $k \in I$. 
Then, we estimate for $k \in I \setminus J$,
\begin{align*}
    X^k_s &= X^k_s - X^k_t + X^k_t \\
    &\leq \max_{j \in I} (X^j_s - X^j_t) + \sum_{j = 1}^N D_{kj} v^{J, k}_t(x^J) \\
    &\leq \sigma \max_{j \in I} (W^j_s - W^j_t) + \sum_{j = 1}^N D_{kj} \bigl(v^{J, j}_s(X^J_s) + C(s - t)\bigr) \\
    &\leq \frac{-\sigma + C D \delta}{\delta} (s - t) + \sum_{j = 1}^N D_{kj} v^{J, j}_s(X^J_s),
\end{align*}
where $D = \max_{\ell \in [N]} \sum_{j = 1}^N D_{\ell j}$. Here we used in the first inequality that $k \in \cal{I}^Iv^{(n - 1)}(t, x)$, in the second inequality that $v^{J, j}$ is $C^{1, 2}$ near $(t, x^J)$, and in the last inequality that $y \mapsto v^{J, j}_s(y)$ is nonincreasing and $X_s \leq x$. But $\frac{-\sigma + CD \delta}{\delta}(s - t)$ is negative for $\delta \in (0, T - t)$ sufficiently small, in which case 
\begin{equation*}
    X^k_s < \sum_{j = 1}^N D_{kj} v^{J, j}_s(X^J_s)
\end{equation*}
for all $k \in I \setminus J$. Since $\delta \in (0, T - t)$ was arbitrary, this proves \eqref{eq:mid_step}. Thus, we have $\limsup_{m \to \infty} \tau_m \leq t$. 

Next, we establish that $\liminf_{m \to \infty} \tau_m \geq t$. Let us choose a subsequence $(m_{\ell})_{\ell \geq 1}$ that achieves the limit inferior, i.e.\@ for which $\lim_{\ell \to \infty} \tau_{m_{\ell}} = \liminf_{m \to \infty} \tau_m = \tau_0$. Note that from above, we know that
\begin{equation*}
    \tau_0 \leq \limsup_{m \to \infty} \tau_m \leq t.
\end{equation*}
Fix $i_{\ell} \in \cal{I}^Iv^{(n - 1)}(\tau_{m_{\ell}}, X^{(m_{\ell})}_{\tau_{m_{\ell}}})$. Descending to a further subsequence if necessary, we may assume that the sequence $(i_{\ell})_{\ell \geq 1}$ is equal to some fixed $i_0 \in I$. Then, appealing to the uniform convergence of $X^{(m)}$ to $X$ and the continuity of $(s, y) \mapsto v^{I \setminus k, j}_s(y^{-k})$ for $k \in I$, we find
\begin{equation*} 
   X^{i_0}_{\tau_0} = \lim_{\ell \to \infty} X^{(m_{\ell}), i_0}_{\tau_{m_{\ell}}} \leq \lim_{\ell \to \infty} \sum_{j = 1}^N D_{i_0j} v^{I \setminus i_0, j}_{\tau_{m_{\ell}}}\bigl(\bigl(X^{(m_{\ell})}_{\tau_{m_{\ell}}}\bigr)^{-i_0}\bigr) = \sum_{j = 1}^N D_{i_0j}  v^{I \setminus i_0, j}_{\tau_0}\bigl((X_{\tau_0})^{-i_0}\bigr), 
\end{equation*}
so that $t \leq \tau_0 = \liminf_{m \to \infty} \tau_m$. Together, we have that $\lim_{m \to \infty} \tau_m = t$. 
From this and the fact that $t < T$, we can also deduce that $\bf{1}_{\{\tau_m \leq T\}} \to 1$. Hence, drawing 
on the continuity of the boundary condition, guaranteed by Lemma \ref{lem:boundary_cond_cont}, and applying the dominated convergence theorem, we find that
\begin{align*}
    \lim_{m \to \infty} v^{I, i}_{t_m}(x^m) &= \lim_{m \to \infty} \ev\Bigl[\F^I_iv^{(n - 1)}\bigl(\tau_m, X^{(m)}_{\tau_m} \bigr) \bf{1}_{\{\tau_m \leq T\}}\Bigr] \\
    &= \F^I_iv^{(n - 1)}\bigl(\tau, X_{\tau}\bigr) \\
    &= v^{I, i}_t(x),
\end{align*}
so that $v^{I, i}$ is continuous on $[0, T) \times \R^I$. Thus, $(v^{I, i})_{i \in I, I \in \P_N^n}$ is a classical solution to PDE \eqref{eq:pde_moving_boundary} up to level $n$. This terminates our construction.

\textit{Uniqueness:} Let $\tilde{v} = (\tilde{v}^{I, i})_{i \in I, I \in \P_N}$ be any classical solution to PDE \eqref{eq:pde_moving_boundary}. Proceeding by induction on the cardinality of $I \in \P_N$, we will show that $\tilde{v}^{I, i} = v^{I, i}$ for all $i \in I$ and $I \in \P_N$, where $v = (v^{I, i})_{i \in I, I \in \P_N}$ is the solution to PDE \eqref{eq:pde_moving_boundary} constructed above. If $I = \varnothing$, there is nothing to show, since $\tilde{v}^{\varnothing, i}_t(0) = 1 = v^{\varnothing, i}_t(0)$ for $t \in [0, T]$ and $i \in [N]$ by definition. Next, suppose that $\tilde{v}^{I, i} = v^{I, i}$ for all $i \in [N]$ and $I \in \P^{n - 1}_N$ and some $n \in \{0, \dots, N - 1\}$. We shall show that the above equality also holds in the case that $\lvert I \rvert = n$. So fix $I \in \P^n_N$ and let $(t, x) \in [0, T] \times \R^I$. For $i \in [N] \setminus I$, we have $\tilde{v}^{I, i}(x) = 1 = v^{I, i}_t(x)$ for all $(t, x) \in [0, T] \times \R^I$, so let us focus on the case $i \in I$. Since $\tilde{v}$ and $v$ coincide up to level $n - 1$, the domain and boundary conditions for the PDEs satisfied by $\tilde{v}^{I, i}$ and $v^{I, i}$ are the same. In particular, if $(t, x) \notin \cal{D}^I_T\tilde{v} = \cal{D}^I_Tv$, then
\begin{equation*}
    \tilde{v}^{I, i}_t(x) = \F^I_i\tilde{v}(t, x) = \F^I_iv(t, x) = v^{I, i}_t(x).
\end{equation*}
Next, suppose that $(t, x) \in \cal{D}^I_T\tilde{v}$ and let $X^{t, x}$ and $\tau_{t, x}$ be as defined above. Note that since $(t, x)$ are inside the domain, we either have that $t = T$, in which case $\tilde{v}^{I, i}_t(x) = 0 = v^{I, i}_t(x)$, or $t \in [0, T)$ and $\tau_{t, x} > t$. In the latter case, it follows from It\^o's formula that for $s \in [t, \tau_{t, x} \land T)$, we have
\begin{align} \label{eq:ito_on_second_sol}
    \tilde{v}^{I, i}_s(X^{t, x}_s) &= \tilde{v}^{I, i}_t(x) + \int_t^s \Bigl(\partial_t \tilde{v}^{I, i}_u(X^{t, x}_u) + \frac{\sigma^2}{2}\Delta\tilde{v}^{I, i}_u(X^{t, x}_u)\Bigr) \, \d u \notag \\
    &\ \ \ + \sum_{j \in I} \int_0^t \sigma \partial_{x_j} \tilde{v}^{I, i}_u(X^{t, x}_u) \, \d W^j_u \notag \\
    &= \tilde{v}^{I, i}_t(x) + \sum_{j \in I} \int_0^t \sigma \partial_{x_j} \tilde{v}^{I, i}_u(X^{t, x}_u) \, \d W^j_u,
\end{align}
where we use in the second equality that $\tilde{v}^{I, i}$ solves the heat equation inside the domain $\cal{D}^I\tilde{v}$ and that $X^{t, x}_u$ lies inside the domain for $u \in [t, \tau_{t, x} \land T)$. 

Define $\tau_m$ to be the first time $s \in [t, T]$ such that $(s, X^{t, x}_s)$ is at distance at most $1/m$ from $([0, T] \times \R^I) \setminus \cal{D}^I_T\tilde{v}$. We claim that $\tau_m \to \tau_{t, x}$ almost surely on $\{\tau_{t, x} < \infty\}$. First note that $\tau_m \leq \tau_{t, x}$ and that the sequence $\tau_m$, $m \geq 1$, is nondecreasing, so that $\tau_m$ has an a.s.\@ limit $\tilde{\tau}$ with values in $[t, \tau_{t, x} \land T] \cup \{\infty\}$. Next, by definition of $\tau_m$, $(\tilde{\tau}, X^{t, x}_{\tilde{\tau}})$ lies outside of $\cal{D}^I_T\tilde{v}$ on $\{\tau_{t, x} < \infty\}$. Since $\tilde{\tau} \leq \tau_{t, x}$ and $\tau_{t, x}$ is the first time at which $(s, X^{t, x}_s)$ lies outside the domain, we can conclude that $\tilde{\tau} = \tau_{t, x}$. Now, we define $\varrho_m$ to be the first time $s \in [t, T]$ that $X^{t, x}_s$ leaves the ball of radius $m$ in $\R^I$ centred at $x$. Clearly, $\varrho_m \to \infty$ as $m \to \infty$. For a fixed $m \geq 1$, all realisations $X^{t, x}_s$ for $s \in [t, T_m]$, where $T_m = \tau_m \land \varrho_m \land (T - \frac{1}{m})$, lie in some compact subset of the domain $\cal{D}^I\tilde{v}$. Since $\tilde{v}^{I, i}$ is an element of $C^{1, 2}(\cal{D}^I\tilde{v})$, it follows that $\partial_{x_j} \tilde{v}^{I, i}_s(X^{t, x}_s)$ is essentially bounded over $s \in [t, T_m]$ and $\omega \in \Omega$. Consequently, evaluating \eqref{eq:ito_on_second_sol} at $s = T_m$ and taking expectation on both sides, we find that $\tilde{v}^{I, i}_t(x) = \ev[\tilde{v}^{I, i}_{T_m}(X^{t, x}_{T_m})]$. Then, taking $m \to \infty$ and using that $\tilde{v}^{I, i}$ is bounded and continuous on $[0, T) \times \R^I$ and that $\pr(\tau_{t, x} = T) = 0$ by Lemma \ref{lem:domain}, we deduce from the dominated convergence theorem that
\begin{align*}
    \tilde{v}^{I, i}_t(x) &= \lim_{m \to \infty}\ev\bigl[\tilde{v}^{I, i}_{T_m}\bigl(X^{t, x}_{T_m}\bigr)\bigr] \\
    &= \lim_{m \to \infty} \ev\Bigl[\tilde{v}^{I, i}_{T_m}\bigl(X^{t, x}_{T_m}\bigr) \bf{1}_{\{\tau_{t, x} \leq T\}}\Bigr] + \lim_{m \to \infty} \ev\Bigl[\tilde{v}^{I, i}_{T_m}\bigl(X^{t, x}_{T_m}\bigr) \bf{1}_{\{\tau_{t, x} > T\}}\Bigr] \\
    &= \ev\Bigl[\tilde{v}^{I, i}_{\tau_{t, x}}\bigl(X^{t, x}_{\tau_{t, x}}\bigr) \bf{1}_{\{\tau_{t, x} \leq T\}}\Bigr] + \ev\Bigl[\tilde{v}^{I, i}_T\bigl(X^{t, x}_T\bigr) \bf{1}_{\{\tau_{t, x} > T\}}\Bigr] \\
    &= \ev\Bigl[\F^I_iv\bigl(\tau_{t, x}, X^{t, x}_{\tau_{t, x}}\bigr) \bf{1}_{\{\tau_{t, x} \leq T\}}\Bigr].
\end{align*}
Here we used in the third equality that if $\tau_{t, x} > T$, then for $\pr$-a.e.\@ $\omega \in \Omega$, the sequence $(X^{t, x}_{T_m})_{m \geq M}$, for some $M = M(\omega)$ large enough, is contained in some compact set $K = K(\omega) \subset \R^I$ such that $\{T\} \times K \subset \cal{D}^I_T\tilde{v}$. Hence, it follows from Property \ref{it:terminal} of the definition of a classical solution that
\begin{equation*}
    \tilde{v}^{I, i}_{T_m}\bigl(X^{t, x}_{T_m}\bigr) \to \tilde{v}^{I, i}_T(X^{t, x}_T) = 0
\end{equation*}
on $\{\tau_{t, x} > T\}$. This concludes the induction and the proof.
\end{proof}

\section{Verification Theorem for PDE  \texorpdfstring{\eqref{eq:pde_moving_boundary}}{(EQ)}} \label{sec:verification}

In this section, we will explain how to obtain a solution to FBSDE \eqref{eq:fbsde_hetero} from the decoupling field $v = (v^{I, i})_{i \in [N], I \in \P_N}$ constructed in Theorem \ref{thm:pde_exist_unique} for any initial data $(\varrho, (\xi_i)_{i \in [N]}, \chi)$. This solution turns out to satisfy a natural flow property, discussed in Proposition \ref{prop:fp}. Fix initial data $(\varrho, (\xi_i)_{i \in [N]}, \chi)$ and let $X^i = (X^i_t)_{t \in [\varrho, T]}$ be given by
\begin{equation} \label{eq:forward_part}
    X^i_t = \xi_i + \sigma (W^i_t - W^i_{\varrho}) 
\end{equation}
for $t \in [\varrho, T]$. We define the following sequences of hitting times and random index sets: set $\varrho_0 = \varrho$ and $\cal{I}_0 = [N]$. Next, assume $\varrho_n$ and $\cal{I}_n$ are given for some $n \in \{0, \dots, N - 1\}$. If $n + 1 \leq N - \lvert \chi\rvert$, set $\varrho_{n + 1} = \varrho$ and $\cal{I}_{n + 1} = \cal{I}_n \setminus i_{n + 1}$, where $i_{n + 1}$ is the smallest index in $\cal{I}_n \setminus \chi$. Otherwise, define
\begin{equation} \label{eq:killing_times_pde_solution}
    \varrho_{n + 1} = \inf\biggl\{t \in [\varrho_n, T] \define X^i_t \leq \sum_{j = 1}^N D_{ij} v^{\cal{I}_n \setminus i, j}_t\bigl(\bf{X}^{\cal{I}_n \setminus i}_t\bigr) \text{ for some } i \in \cal{I}_n\biggr\},
\end{equation}
where $\bf{X} = (X^1, \dots, X^N)$. Then, if $\varrho_n \leq T$, we set $\cal{I}_{n + 1} = \cal{I}_n \setminus i_{n + 1}$, where $i_{n + 1}$ is the smallest index $i$ in $\cal{I}_n$ with
\begin{equation*}
    X^i_{\varrho_{n + 1}} \leq \sum_{j = 1}^N D_{ij} v^{\cal{I}_n \setminus i, j}_{\varrho_{n + 1}}\bigl(\bf{X}^{\cal{I}_n \setminus i}_{\varrho_{n + 1}}\bigr).
\end{equation*}
Otherwise, we set $\cal{I}_{n + 1} = \cal{I}_n$. This completes the construction. 

The above procedure can be summarised as follows: we begin with a complete system $[N]$ of alive particles. Then, we first successively remove particles from $[N]$ until only those in the set $\chi$ are left. Next, we run the target system, iteratively eliminating particles when they hit the killing threshold. Once the final time $T$ is reached, we set the remaining hitting times to infinity and the set of alive particles becomes stationary. This process ensures that $\lvert \cal{I}_n\rvert = N - n$ if $\varrho_n < \infty$.

Now, let us define the random time-varying index set $\bf{I} = (\bf{I}_t)_{t \in [\varrho, T]}$ of living particles by $\bf{I}_t = \cal{I}_n$ if $t \in [\varrho_n, \varrho_{n + 1})$ for some $n \in \{0, \dots, N - 1\}$ and $\bf{I}_t = \varnothing$ if $t \in [\varrho_N, T]$. With this, we have
\begin{equation} \label{eq:in_domain}
    \bigl(t, \bf{X}^{\bf{I}_t}_t\bigr) \in \cal{D}^{\bf{I}_t}_T
\end{equation}
for all $t \in [\varrho, T]$. Then, we define the processes $Y^i = (Y^i_t)_{t \in [\varrho, T]}$ and $Z^i = (Z^i_t)_{t \in [\varrho, T]}$, $i \in [N]$ by
\begin{equation} \label{eq:y_and_z_decoupling}
    Y^i_t = v^{\bf{I}_t, i}_t\bigl(\bf{X}^{\bf{I}_t}_t\bigr) \quad \text{and} \quad Z^{ij}_t = - \sigma \partial_{x_j} v^{\bf{I}_t, i}_t\bigl(\bf{X}^{\bf{I}_t}_t\bigr)
\end{equation}
for $t \in [0, T]$ and $j \in [N]$. Note that $Z^{ij}_t$ vanishes if $j \notin \bf{I}_t$, since then $v^{\bf{I}_t, i}$ does not depend on the $j$th coordinate. We also introduce the stopping time $\tau_i$ given by $\tau_i = \varrho$ on $\{i \notin \chi\}$ and
\begin{equation} \label{eq:particle_killing_pde_sol}
    \tau_i = \inf\biggl\{t \in [\varrho, T] \define X^i_t \leq \sum_{j = 1}^N D_{ij} Y^j_t\biggr\}
\end{equation}
on $\{i \in \chi\}$. The definition of $\tau_i$ implies that $\tau_i = \varrho$ if $i \notin \chi$, $\tau_i = \varrho_n$ if $\varrho_n < \infty$ and $i = i_n$, and $\tau_i = \infty$ otherwise. Whenever we want to emphasise the dependence of $X^i$, $Y^i$, $Z^i$, and $\bf{I}$ on the initial data $D = (\varrho, (\xi_i)_{i \in [N]}, \chi)$, we write $X^{D, i}$, $Y^{D, i}$, $Z^{D, i}$, and $\bf{I}^D$ instead of $X^i$, $Y^i$, $Z^i$, and $\bf{I}$.

\begin{theorem} \label{thm:verification}
Let $v = (v^{I, i})_{i \in [N], I \in \P_N}$ be the classical solution of the system of PDEs \eqref{eq:pde_moving_boundary} and for initial data $(\varrho, (\xi_i)_{i \in [N]}, \chi)$, define $X^i$, $Y^i$ and $Z^i$, $i \in [N]$, as in \eqref{eq:forward_part} and \eqref{eq:y_and_z_decoupling}, respectively. Then $(X^i, Y^i, Z^i)_{i \in [N]}$ is a solution to FBSDE \eqref{eq:fbsde_hetero}. 
\end{theorem}

\begin{proof}
We have to verify the following four statements for all $i \in [N]$: $Y^i \in \bb{S}^2_{\varrho, T}$, $Z^i \in \bb{H}^{2, N}_{\varrho, T}$,
\begin{equation*}
    Y^i_t = Y^i_{\varrho} + \int_{\varrho}^t Z^i_s \cdot \d \bf{W}_s
\end{equation*}
for $t \in [\varrho, T]$, and $Y^i_T = \bf{1}_{\{\tau_i \leq T\}}$. Let us begin with the first one. Since $v$ is bounded by Theorem \ref{thm:pde_exist_unique}, to obtain $Y^i \in \bb{S}^2_{\varrho, T}$ it is enough to prove that $Y^i$ is a.s.\@ continuous on $[\varrho, T]$. This is clear for $t \in [\varrho_n, \varrho_{n + 1}) \cap [\varrho, T)$, $n \in \{0, \dots, N - 1\}$, from the continuity of $v^{I, i}$ on $[0, T) \times \R^I$. Hence, we have to show left-continuity at $\varrho_n$ for $n = 1$,~\ldots, $N$ on $\{\varrho < \varrho_n < T\}$ and left-continuity at $T$. For the former, we simply note that on $\{\varrho < \varrho_n < T\}$, we have
\begin{align*}
    \lim_{t \nearrow \varrho_n} Y^i_t = \lim_{t \nearrow \varrho_n} v^{\cal{I}_{n - 1}, i}_t\bigl(\bf{X}^{\cal{I}_{n - 1}}_t\bigr) = v^{\cal{I}_{n - 1}, i}_{\varrho_n}\bigl(\bf{X}^{\cal{I}_{n - 1}}_{\varrho_n}\bigr) = v^{\cal{I}_n, i}_{\varrho_n}\bigl(\bf{X}^{\cal{I}_n}_{\varrho_n}\bigr) = Y^i_{\varrho_n},
\end{align*}
where we used in the second equality that $v^{I, i}$ is continuous on $[0, T) \times \R^I$ and in the third that $(\varrho_n, \bf{X}^{\cal{I}_{n - 1}}_{\varrho_n}) \in ([0, T) \times \R^{\cal{I}_{n - 1}}) \setminus \cal{D}^{\cal{I}_{n - 1}}v$, so that the boundary condition of PDE \eqref{eq:pde_moving_boundary} applies. To address the left-continuity at $T$, we first note that if $\varrho_N < T$, then $Y^i$ is constant on $[\varrho_N, T]$, so in particular continuous. Hence, we may assume that $\varrho_N \geq T$. Next, we observe that
\begin{equation} \label{eq:no_killing_at_terminal}
    \pr(\varrho_n = T,\, \varrho < T) = \pr\biggl(X^i_T = \sum_{j = 1}^N \bf{1}_{\tau_j \leq T} \text{ for some } i \in [N],\, \varrho < T\biggr) = 0,
\end{equation}
since the subprobability distribution of $X^i_T = (X^i_T - X^i_{\varrho}) + X^i_{\varrho}$ on the event $\{\varrho < T\}$ has a density with respect to the Lebesgue measure and the set of elements of the form $\sum_{j = 1}^N D_{ij} w_j$, for $w = (w_1, \dots, w_N) \in \{0, 1\}^N$, is finite. Thus, we may in fact assume that $\varrho_N > t$, in which case $T$ lies in one of the intervals $[\varrho_n, \varrho_{n + 1})$ for $n \in \{0, \dots, N - 1\}$. Hence, the left-continuity at $T$ follows from the continuity of $Y^i$ on those intervals. Thus, $Y^i$ is continuous on all of $[\varrho, T]$.

Next, we show that $Z^i \in \bb{H}^{2, N}_{\varrho, T}$ and that $Y^i_t = Y^i_{\varrho} + \int_{\varrho}^t Z^i_s \cdot \d \bf{W}_s$ holds for $t \in [\varrho, T]$. For $n \in \{0, \dots, N - 1\}$ on $\{\varrho_n < \varrho_{n + 1}\}$, we apply It\^o's formula for $t \in [\varrho_n, \varrho_{n + 1} \land T)$, whereby
\begin{align} \label{eq:martingale_pre}
    Y^i_t &= v^{\cal{I}_N, i}_t\bigl(\bf{X}^{\cal{I}_N}_t\bigr) \notag \\
    &= v^{\cal{I}_N, i}_{\varrho_n}\bigl(\bf{X}^{\cal{I}_N}_{\varrho_n}\bigr) + \int_{\varrho_n}^t \biggl(\partial_s v^{\cal{I}_N, i}_s\bigl(\bf{X}^{\cal{I}_N}_s\bigr) + \frac{\sigma^2}{2} \Delta v^{\cal{I}_N, i}_s\bigl(\bf{X}^{\cal{I}_N}_s\bigr)\biggr) \, \d s \notag \\
    &\ \ \ + \sum_{j \in \cal{I}_N}\int_{\varrho_n}^t \sigma \partial_{x_j} v^{\cal{I}_N, i}_s\bigl(\bf{X}^{\cal{I}_N}_s\bigr) \, \d W^j_s \notag \\
    &= Y^i_{\varrho_n} - \sum_{j \in \cal{I}_N} \int_{\varrho_n}^t Z^{ij}_s \, \d W^j_s.
\end{align}
Note that the finite variation term in the second line above vanishes, since on $\{\varrho_n < \varrho_{n + 1}\}$, the function $v^{\cal{I}_N, i}$ solves the heat equation in the domain $\cal{D}^{\cal{I}_n}v$ and $(s, \bf{X}^{\cal{I}_N}_s)$ lies in this domain for $s \in [\varrho_n, t)$. Now, define
\begin{equation*}
    \varrho_{n + 1, \epsilon} = \inf\biggl\{t \in [\varrho_n, T] \define X^k_t \leq \sum_{j = 1}^N D_{kj} v^{\cal{I}_n \setminus k, j}_t\bigl(\bf{X}^{\cal{I}_n \setminus k}_t\bigr) + \epsilon \text{ for some } k \in \cal{I}_n\biggr\}
\end{equation*}
for $\epsilon > 0$. Then we deduce from the continuity of $v^{J, j}$ on $[0, T) \times \R^J$ for $j \in [N]$ and $J \in \P_N$ that $\varrho_{n + 1, \epsilon} < \varrho_{n + 1}$, $\lim_{\epsilon \to 0} \varrho_{n + 1, \epsilon} = \varrho_{n + 1}$, and $(\varrho_{n + 1, \epsilon}, \bf{X}^{\cal{I}_N}_{\varrho_{n + 1, \epsilon}}) \in \cal{D}^{\cal{I}_n}v$ on $\{\varrho_{n + 1} < \infty\}$. Consequently, evaluating \eqref{eq:martingale_pre} at $t = \varrho_{n + 1, \epsilon} \land T$, we deduce from It\^o's isometry that
\begin{align*}
    \ev\bigl\lvert Y^i_{\varrho_{n + 1} \land T} - Y^i_{\varrho_n \land T}\bigr\rvert^2 &= \lim_{\epsilon \to 0} \ev\bigl\lvert Y^i_{\varrho_{n + 1, \epsilon} \land T} - Y^i_{\varrho_n \land T}\bigr\rvert^2 \\
    &= \lim_{\epsilon \to 0} \ev\biggl[\int_{\varrho_n \land T}^{\varrho_{n + 1, \epsilon} \land T} \lvert Z^i_s\rvert^2 \, \d s\biggr] \\
    &= \ev\biggl[\int_{\varrho_n \land T}^{\varrho_{n + 1} \land T} \lvert Z^i_s\rvert^2 \, \d s\biggr].
\end{align*}
Here we applied the dominated convergence theorem in the first step and the monotone convergence theorem in the last equality. Summing the above equality over $n \in \{0, \dots, N - 1\}$ shows that $\ev \int_{\varrho}^{\varrho_N \land T} \lvert Z^i_s\rvert^2 \, \d s = \ev \int_{\varrho}^T \lvert Z^i_s\rvert^2 \, \d s$ is finite. Thus, $Z^i \in \bb{H}^{2, N}_{\varrho, T}$. Moreover from \eqref{eq:martingale_pre} and the continuity of $Y^i$, we deduce that $Y^i_t = Y^i_{\varrho} + \int_{\varrho}^t Z^i_s \cdot \d \bf{W}_s$ for $t \in [\varrho, T]$. 

Lastly, to verify the terminal condition of FBSDE \eqref{eq:fbsde_hetero}, note that by \eqref{eq:in_domain}, the terminal condition of PDE \eqref{eq:pde_moving_boundary}, and the fact that $v^{I, i}_t(x) = 1$ for $(t, x) \in [0, T] \times \R^I$ whenever $i \notin I \in \P_N$, we have
\begin{equation*}
    Y^i_T = v^{\bf{I}_T}_T\bigl(\bf{X}^{\bf{I}_T}_T\bigr) = \bf{1}_{\{i \notin \bf{I}_T\}}.
\end{equation*}
However, it clearly holds that $i \notin \bf{I}_T$ if and only if $\tau_i \leq T$, so that $Y^i_T = \bf{1}_{\{\tau_i \leq T\}}$. This completes the proof.
\end{proof}

The solution constructed in \eqref{eq:y_and_z_decoupling}, which we verified in Theorem \ref{thm:verification}, satisfies the following flow property. Note that this flow property is analogous to the one derived for the maximal solution to FBSDE \eqref{eq:fbsde_hetero} in Corollary \ref{cor:flow_property}.

\begin{proposition} \label{prop:fp}
Let $D_k = (\varrho_k, (\xi^k_i)_{i \in [N]}, \chi_k)$, $k = 1$, $2$, be initial data such that $\varrho_1 \leq \varrho_2$, $\xi^2_i = X^{D_1, i}_{\varrho_2}$ for $i \in [N]$, and $\chi_2 = \bf{I}^{D_1}_{\varrho_2}$. Then a.s.\@
\begin{equation*}
    \bigl(X^{D_1, i}_t, Y^{D_1, i}_t, Z^{D_1, i}_t\bigr)_{i \in [N]} = \bigl(X^{D_2, i}_t, Y^{D_2, i}_t, Z^{D_2, i}_t\bigr)_{i \in [N]}
\end{equation*}
for all $t \in [\varrho_2, T]$ and $i \in [N]$.
\end{proposition}

\begin{proof}
Let $\varrho^k_n$, $n = 0$,~\ldots, $N$, denote the hitting times defined in and above \eqref{eq:killing_times_pde_solution} for the solution $(X^{D_k, i}, Y^{D_k, i}, Z^{D_k, i})_{i \in [N]}$. We only have to show that
\begin{equation*}
    Y^{D_1, i}_t = Y^{D_2, i}_t
\end{equation*}
for all $t \in [\varrho_2, T]$ and $i \in [N]$, since the identity $Z^{D_1, i}_t = Z^{D_2, i}_t$ then follows from the uniqueness part of the martingale representation theorem. Now, on $\{\varrho_1 = \varrho_2\}$, the equality $Y^{D_1, i}_t = Y^{D_2, i}_t$ follows immediately from the construction in \eqref{eq:y_and_z_decoupling}. On $\{\varrho_1 < \varrho_2\}$, we can either find $n \in \{0 ,\dots, N - 1\}$ such that $\varrho_2 \in [\varrho^1_n, \varrho^1_{n + 1})$ or $\varrho_2 \in [\varrho^1_N, T)$. In the latter case, $\chi_2 = \bf{I}^{D_1}_{\varrho_2} = \varnothing$, so that
\begin{equation*}
    Y^{D_1, i}_t = 1 = Y^{D_2, i}_t
\end{equation*}
for all $t \in [\varrho_2, T]$ and $i \in [N]$. Hence, let us assume that $\varrho_2 \in [\varrho^1_n, \varrho^1_{n + 1})$ for some $n \in \{0 ,\dots, N - 1\}$. Then it is easy to verify that from $\varrho_2$ onwards, the construction of $Y^{D_1, i}_t$ and $Y^{D_2, i}_t$ for $i \in [N]$ is completely identical, so these processes must coincide. This concludes the proof.
\end{proof}

\section{Uniqueness for FBSDE \texorpdfstring{\eqref{eq:fbsde_hetero}}{(EQ)}} \label{sec:uniqueness_fbsde}

Building on the results from the previous two section, we can finally prove the uniqueness of FBSDE \eqref{eq:fbsde_hetero}. We only require the following auxiliary result. Despite its simplicity, it turns out to be crucial for the proof of the uniqueness theorem.

\begin{lemma} \label{lem:hitting_time}
Let $S > 0$ and $f^i_{\pm} \define [0, S] \to \R$, $i \in [n]$, be continuous functions such that $f^i_-(0) \leq 0 < f^i_+(0)$. Next, define
\begin{equation*}
    \tau_{\pm} = \inf\Bigl\{t \in [0, S] \define W^i_t = f^i_{\pm}(t) \textup{ for some } i \in [n]\Bigr\}.
\end{equation*}
Then $\pr(\tau_- < \tau_+) > 0$.
\end{lemma}

\begin{proof}
Set $b_{\pm} = \frac{1}{2} \min_{i \in [n]} f^i_{\pm}(0)$ and let  $t_0$ be the infimum over all times $t \in [0, S]$ such that $\min_{i \in [n]} f^i_-(t) \leq b_-$ or $\min_{i \in [n]} f^i_+(t) \leq b_+$. Next, we set
\begin{equation*}
    \tilde{\tau}_{\pm} = \inf\Bigl\{t \in [0, t_0] \define W^i_t = b_{\pm} \text{ for some } i \in [n]\Bigr\}.
\end{equation*}
Then it follows that $\pr(\tilde{\tau}_- < \tilde{\tau}_+) \leq \pr(\tau_- < \tau_+)$. However, we clearly have that $\pr(\tilde{\tau}_- < \tilde{\tau}_+)$ is positive, which then implies $\pr(\tau_- < \tau_+) > 0$.
\end{proof}

The uniqueness proof proceeds by inductively linking an arbitrary solution to the solution constructed from the decoupling field in Section \ref{sec:verification}, moving backwards from the final time. The main challenge lies in verifying that if we start an arbitrary solution and the solution based on the decoupling field from the same initial data, then the first killing occurs at the same time in both systems. This is where Lemma \ref{lem:hitting_time} comes in handy.

\begin{theorem} \label{thm:uniqueness}
FBSDE \eqref{eq:fbsde_hetero} has a unique solution for any initial data $(\varrho, (\xi_i)_{i \in [N]}, \chi)$.
\end{theorem}

\begin{proof}
Let $(X^i, \tilde{Y}^i, \tilde{Z}^i)_{i \in [N]}$ be a solution to FBSDE \eqref{eq:fbsde_hetero} with initial data $(\varrho, (\xi_i)_{i \in [N]}, \chi)$. Similarly to above Theorem \ref{thm:verification}, let us define a sequence of hitting times and random index sets. We set $\tilde{\varrho}_0 = \varrho$ and $\tilde{\cal{I}}_0 = [N]$. Next, assuming $\tilde{\varrho}_n$ and $\tilde{\cal{I}}_n$ are given for some $n \in \{0, \dots, N - 1\}$, we set $\tilde{\varrho}_{n + 1} = \varrho$ and $\tilde{\cal{I}}_{n + 1} = \tilde{\cal{I}}_n \setminus i_{n + 1}$ if $n + 1 \leq N - \lvert \chi\rvert$, where $i_{n + 1}$ is the smallest index in $\tilde{\cal{I}}_n \setminus \chi$. Otherwise, we define
\begin{equation}
    \tilde{\varrho}_{n + 1} = \inf\biggl\{t \in [\tilde{\varrho}_n, T] \define X^i_t \leq \sum_{j = 1}^N D_{ij} \tilde{Y}^j_t \text{ for some } i \in \tilde{\cal{I}}_n \biggr\}.
\end{equation}
If $\tilde{\varrho}_n \leq T$, we set $\tilde{\cal{I}}_{n + 1} = \tilde{\cal{I}}_n \setminus i_{n + 1}$, where $i_{n + 1}$ is the smallest index $i \in \tilde{\cal{I}}_n$ with
\begin{equation*}
    X^i_{\tilde{\varrho}_{n + 1}} \leq \sum_{j = 1}^N D_{ij} \tilde{Y}^j_{\tilde{\varrho}_{n + 1}}.
\end{equation*}
Otherwise, we set $\tilde{\cal{I}}_{n + 1} = \tilde{\cal{I}}_n$. Lastly, define the index set $\tilde{\bf{I}}_t = \{i \in [N] \define \tau_i > t\}$ for $t \in [\varrho, T]$. 

Next, for $n \in \{0, \dots, N\}$, we let $(X^i, Y^{(n), i}, Z^{(n), i})_{i \in [N]}$ denote the solution to FBSDE \eqref{eq:fbsde_hetero} constructed above Theorem \ref{thm:verification} from the unique solution $v$ of PDE \eqref{eq:pde_moving_boundary} with initial data
\begin{equation*}
    \bigl(\tilde{\varrho}_n \land T, (X^i_{\tilde{\varrho}_n \land T})_{i \in [N]}, \tilde{\cal{I}}_n\bigr).
\end{equation*}
Note that on $\{\tilde{\varrho}_n < \infty\}$, we have that $\lvert \tilde{\cal{I}}_n\rvert = N - n$. Denote the corresponding sequential killing times and random index sets defined in and above \eqref{eq:killing_times_pde_solution} by $\varrho^n_k$ and $\cal{I}^n_k$, $k \in \{0, \dots, N\}$, the particle killing times defined in \eqref{eq:particle_killing_pde_sol} by $\tau^n_i$, $i \in [N]$, and define the time-varying index set $\bf{I}^n_t = \{i \in [N] \define \tau^n_i > t\}$ for $t \in [\tilde{\varrho}_n \land T, T]$. Since $\lvert \tilde{\cal{I}}_n\rvert = N - n$ when $\tilde{\varrho}_n < \infty$, it follows from the construction of $(X^i, Y^{(n), i}, Z^{(n), i})_{i \in [N]}$ that $\varrho^n_k = \tilde{\varrho}_n$ on $\{\tilde{\varrho}_n < \infty\}$ for $k \in \{0, \dots, n\}$. Our goal is to inductively show that for all $n \in \{0, \dots, N\}$, we have
\begin{equation*}
    \tilde{Y}^i_t = Y^{(n), i}_t
\end{equation*}
for $t \in [\tilde{\varrho}_n, T]$ and $i \in [N]$, where we recall the convention $[\infty, T] = \varnothing$. Since $\tilde{\varrho}_0 = \varrho$ and $\tilde{\cal{I}}_0 = \chi$, applying this result for $n = 0$ implies that $\tilde{Y}^i_t = Y^{(0), i}_t$ for $t \in [\varrho, T]$, which uniquely determines the process $\tilde{Y}^i$ for $i \in [N]$. The uniqueness part of the martingale representation theorem then implies that $\tilde{Z}^i$, $i \in [N]$, is uniquely determined, so that the solution $(X^i, \tilde{Y}^i, \tilde{Z}^i)_{i \in [N]}$ is unique. 

\textit{Induction:} Let $n = N$. If $\tilde{\varrho}_N = \infty$, there is nothing to show. If $\tilde{\varrho}_N < \infty$, then $\tilde{\cal{I}}_N = \varnothing$, so that $\tilde{Y}^i_t = 1 = Y^{(N), i}_t$ for all $t \in [\tilde{\varrho}_N, T]$ and $i \in [N]$. Next, suppose that the induction statement holds for some $n \in [N]$. Let us claim for now without proof that $\tilde{\varrho}_n = \varrho^{n - 1}_n$ and $\tilde{\cal{I}}_n = \cal{I}^{n - 1}_n$. If this is true then on $\{\tilde{\varrho}_n < \infty\}$, we have by Proposition \ref{prop:fp} that $Y^{(n), i}_t = Y^{(n - 1), i}_t$ and by the induction hypothesis that $\tilde{Y}^i_t = Y^{(n), i}_t$ for $t \in [\tilde{\varrho}_n, T]$. Combining these two statements implies that 
\begin{equation*}
    \tilde{Y}^i_t = Y^{(n - 1), i}_t
\end{equation*}
for $t \in [\tilde{\varrho}_n, T]$. Next, on $\{\tilde{\varrho}_{n - 1} < \tilde{\varrho}_n = \infty\}$, it follows from the fact that $\varrho^{n - 1}_n = \tilde{\varrho}_n = \infty$ that neither in the system $(X^i, \tilde{Y}^i, \tilde{Z}^i)_{i \in [N]}$ nor in the system $(X^i, Y^{(n - 1), i}, Z^{(n - 1), i})_{i \in [N]}$ a particle is killed between $\tilde{\varrho}_{n - 1}$ and $T$. Thus, both $\tilde{Y}^i_T$ and $Y^{(n - 1), i}_T$ vanish if $i \in \tilde{\cal{I}}_{n - 1}$ and are equal to one otherwise. Consequently, we showed that on $\{\tilde{\varrho}_{n - 1} < \infty\}$ it holds that
\begin{equation*}
    \tilde{Y}^i_{\tilde{\varrho}_n \land T} = Y^{(n - 1), i}_{\tilde{\varrho}_n \land T}.
\end{equation*}
Now, let $\tau$ be any stopping time such that $\tilde{\varrho}_{n - 1} \leq \tau \leq \tilde{\varrho}_n$. Then, since both $\tilde{Y}^i$ and $Y^{(n - 1), i}$ are martingales and $\{\tilde{\varrho}_{n - 1} < \infty\}$ is $\F_{\tau}$-measurable, we have
\begin{equation*}
    0 = \ev\bigl[\bf{1}_{\{\tilde{\varrho}_{n - 1} < \infty\}} \bigl(\tilde{Y}^i_{\tilde{\varrho}_n \land T} - Y^{(n - 1), i}_{\tilde{\varrho}_n \land T}\bigr) \big\vert \F_{\tau}\big] = \bf{1}_{\{\tilde{\varrho}_{n - 1} < \infty\}} \bigl(\tilde{Y}^i_{\tau \land T} - Y^{(n - 1), i}_{\tau \land T}\bigr).
\end{equation*}
Since $\tau$ was arbitrary with $\tilde{\varrho}_{n - 1} \leq \tau \leq \tilde{\varrho}_n$ and both $\tilde{Y}^i$ and $Y^{(n - 1), i}$ have a.s.\@ continuous trajectories, we deduce that $\tilde{Y}^i_t = Y^{(n - 1), i}_t = 0$ for $t \in [\tilde{\varrho}_{n - 1}, T]$. Here we are implicitly using that $[\tilde{\varrho}_{n - 1}, T]$ is nonempty only on $\{\tilde{\varrho}_{n - 1} < \infty\}$. This concludes the induction step. It remains to establish our earlier claim that $\tilde{\varrho}_n = \varrho^{n - 1}_n$ and $\tilde{\cal{I}}_n = \cal{I}^{n - 1}_n$. This turns out to be the main intricacy of the proof.

\textit{Proof of} $\tilde{\varrho}_n = \varrho^{n - 1}_n$\textit{:} First, let us prove that $\varrho_n \leq \tilde{\varrho}_n$, where we set $\varrho_n = \varrho^{n - 1}_n$ for notational simplicity. If $\tilde{\varrho}_n = \infty$, this is trivial, so let us suppose that $\tilde{\varrho}_n < \infty$. By the induction hypothesis, we have that $\tilde{Y}^i$ coincides with $Y^{(n), i}$ on $[\tilde{\varrho}_n, T]$. Let $i$ be the unique element in $\tilde{\cal{I}}_{n - 1} \setminus \tilde{\cal{I}}_n$. Since the corresponding particle is killed at time $\tilde{\varrho}_n$, it holds that
\begin{equation} \label{eq:dead_at_step_n}
    X^i_{\tilde{\varrho}_n} \leq \sum_{j = 1}^N D_{ij} \tilde{Y}^j_{\tilde{\varrho}_n} = \sum_{j = 1}^N D_{ij} v^{\tilde{\cal{I}}_{n - 1} \setminus i, j}_{\tilde{\varrho}_n}\Bigl(\bf{X}^{\tilde{\cal{I}}_{n - 1} \setminus i}_{\tilde{\varrho}_n}\Bigr).
\end{equation}
However, by definition of $\varrho_n$, we have
\begin{equation*}
    \varrho_n = \inf\biggl\{t \in [\tilde{\varrho}_{n - 1}, T] \define X^k_t \leq \sum_{j = 1}^N D_{kj} v^{\tilde{\cal{I}}_{n - 1} \setminus k, j}_t\Bigl(\bf{X}^{\tilde{\cal{I}}_{n - 1} \setminus k}_t\Bigr) \text{ for some } k \in \tilde{\cal{I}}_{n - 1} \biggr\} \leq \tilde{\varrho}_n.
\end{equation*}
Thus, we can conclude that $\varrho_n \leq \tilde{\varrho}_n$.

To deduce equality between $\varrho_n$ and $\tilde{\varrho}_n$, it remains to show that $\pr(\varrho_n < \tilde{\varrho}_n) = 0$. Suppose that this is not the case. As we demonstrate below, this implies $\inf_{\varrho_n \leq t \leq \tilde{\varrho}_n} X^i_t < 0$ for some $i \in \tilde{\cal{I}}_{n - 1}$ with positive probability. But the latter leads to a contradiction since $\tilde{\varrho}_n$ must occur before the first time that one of the processes $X^i_t$, for $i \in \tilde{\cal{I}}_{n - 1}$, hits the origin. So to proceed, let us establish that $\pr(\varrho_n < \tilde{\varrho}_n) > 0$ implies $\inf_{\varrho_n \leq t \leq \tilde{\varrho}_n} X^i_t < 0$ for some $i \in \tilde{\cal{I}}_{n - 1}$ with positive probability. Note that on $\{\varrho_n < \tilde{\varrho}_n\}$, we have $\tilde{\varrho}_{n - 1} \leq \varrho_n < \tilde{\varrho}_n$, so owing to the continuity of $\bf{X}$ and $(\tilde{Y}^1, \dots, \tilde{Y}^N)$, \eqref{eq:dead_at_step_n} holds with equality. Now, define the random index set $\cal{D}$ to consist of those $i \in \tilde{\cal{I}}_{n - 1}$ such that
\begin{equation*}
    X^i_{\varrho_n} \leq \sum_{j = 1}^N D_{ij} v^{\tilde{\cal{I}}_{n - 1} \setminus i, j}_{\varrho_n}\Bigl(\bf{X}^{\tilde{\cal{I}}_{n - 1} \setminus i}_{\varrho_n}\Bigr)
\end{equation*}
if $\varrho_n < \infty$ and $\cal{D} = \varnothing$ if $\varrho_n = \infty$. Set $\cal{A} = \tilde{\cal{I}}_{n - 1} \setminus \cal{D}$. On $\{\varrho_n < \infty\}$, we have for all $t \in [\varrho_n, T]$ sufficiently close to $\varrho_n$ by the construction of $(X^i, Y^{(n - 1), i}, Z^{(n - 1), i})_{i \in [N]}$ that
\begin{equation*}
    Y^{(n - 1), i}_t = v^{\cal{A}, i}_t\bigl(\bf{X}^{\cal{A}}_t\bigr)
\end{equation*}
for $i \in [N]$. Since the increments $X^i_{\varrho_n + t} - X^i_{\varrho_n}$, $i \in \tilde{\cal{I}}_{n - 1}$, are independent conditional on $\{\varrho_n < \infty\}$ and the function $v^{\cal{A}, i}$ is continuous, it follows that with positive probability there exists $t \in [\varrho_n, \tilde{\varrho}_n)$ such that such that
\begin{equation*}
    X^k_t < \sum_{j = 1}^N D_{kj} v^{\cal{A}, j}_t\bigl(\bf{X}^{\cal{A}}_t\bigr) \quad \text{and} \quad X^{\ell}_t > \sum_{j = 1}^N D_{{\ell}j} v^{\cal{A}, j}_t\bigl(\bf{X}^{\cal{A}}_t\bigr)
\end{equation*}
for all $k \in \cal{D}$ and $\ell \in \cal{A}$. Consequently, letting $\tau$ denote the infimum over all $t \in [\varrho_n, T]$ such that
\begin{equation*}
    X^k_t \leq \sum_{j = 1}^N D_{kj} v^{\cal{A}, j}_t\bigl(\bf{X}^{\cal{A}}_t\bigr) - \epsilon \quad \text{and} \quad X^{\ell}_t \geq \sum_{j = 1}^N D_{{\ell}j} v^{\cal{A}, j}_t\bigl(\bf{X}^{\cal{A}}_t\bigr) + \epsilon
\end{equation*}
for all $k \in \cal{D}$ and $\ell \in \cal{A}$, we have $\pr(\tau < \tilde{\varrho}_n) > 0$ for all sufficiently small $\epsilon > 0$. Next, let us define the processes
\begin{equation*}
    \tilde{\bf{X}}_t = \bf{X}_{(\tau \land T) + t} \quad \text{and} \quad \bf{B}_t = \biggl(\sum_{j = 1}^N D_{ij} v^{\cal{A}, j}_{(\tau \land T) + t}\Bigl(\bf{X}^{\cal{A}}_{(\tau \land T) + t}\Bigr)\biggr)_{i \in [N]}
\end{equation*}
for $t \in [0, (T - \tau)_+]$ as well as the random time $\varsigma$ as the minimum between $(T - \tau)_+$ and the infimum over all $t \in [0, (T - \tau)_+]$ such that $\tilde{\bf{X}}^{\cal{A}}_t \leq \bf{B}^{\cal{A}}_t$. Finally, let us define 
\begin{equation*}
    \tau_- = \inf\bigl\{t \in [0, \varsigma] \define \tilde{\bf{X}}^{\cal{D}}_t \leq -1 \bigr\} \quad \text{and} \quad \tau_+ = \inf\bigl\{t \in [0, \varsigma] \define \tilde{\bf{X}}^{\cal{D}}_t \geq \bf{B}^{\cal{D}}_t\bigr\}.
\end{equation*}
Since \eqref{eq:dead_at_step_n} holds with equality if $\tau < \tilde{\varrho}_n$, we have that $\tau + (\tau_+ \land \varsigma) \leq \tilde{\varrho}_n$ on $\{\tau < \tilde{\varrho}_n\}$. Indeed, suppose otherwise, that $\tilde{\varrho}_n$ occurs before $\tau + (\tau_+ \land \varsigma)$. Then by \eqref{eq:dead_at_step_n}, we have
\begin{equation} \label{eq:dead_at_step_n_equality}
    X^i_{\tilde{\varrho}_n} = \sum_{j = 1}^N D_{ij} v^{\tilde{\cal{I}}_{n - 1} \setminus i, j}_{\tilde{\varrho}_n}\Bigl(\bf{X}^{\tilde{\cal{I}}_{n - 1} \setminus i}_{\tilde{\varrho}_n}\Bigr)
\end{equation}
for some $i \in \tilde{\cal{I}}_{n - 1}$. If $i \in \cal{A}$, it follows from the boundary condition of PDE \eqref{eq:pde_moving_boundary}, Lemma \ref{lem:mon_well} \ref{it:monotonicity}, and the assumption that $\tilde{\varrho}_n < \tau + \varsigma$ that
\begin{equation*}
    X^i_{\tilde{\varrho}_n} = \sum_{j = 1}^N D_{ij} v^{\tilde{\cal{I}}_{n - 1} \setminus i, j}_{\tilde{\varrho}_n}\Bigl(\bf{X}^{\tilde{\cal{I}}_{n - 1} \setminus i}_{\tilde{\varrho}_n}\Bigr) = \sum_{j = 1}^N D_{ij} v^{\tilde{\cal{I}}_{n - 1}, j}_{\tilde{\varrho}_n}\Bigl(\bf{X}^{\tilde{\cal{I}}_{n - 1}}_{\tilde{\varrho}_n}\Bigr) \leq \sum_{j = 1}^N D_{ij} v^{\cal{A}, j}_{\tilde{\varrho}_n}\bigl(\bf{X}^{\cal{A}}_{\tilde{\varrho}_n}\bigr) < X^i_{\tilde{\varrho}_n}.
\end{equation*}
Thus, it should hold that $i \in \cal{D}$. We will show that this cannot be the case either. Indeed, since $\tau + \tau_+$ did not occur before or at $\tilde{\varrho}_n$, we have
\begin{equation} \label{eq:plus_not_occured}
    X^k_{\tilde{\varrho}_n} < \sum_{j = 1}^N D_{kj} v^{\cal{A}, j}_t\bigl(\bf{X}^{\cal{A}}_{\tilde{\varrho}_n}\bigr)
\end{equation}
for all $k \in \cal{D}$. Hence, the only if direction of Lemma \ref{lem:mult_default} implies that $\cal{D} \subset \cal{I}^{\tilde{\cal{I}}_{n - 1}}\bigl(\tilde{\varrho}_n, \bf{X}^{\tilde{\cal{I}}_{n - 1}}_{\tilde{\varrho}_n}\bigr)$. Thus, Equation \eqref{eq:bc_down} from Lemma \ref{lem:mult_default} shows that 
\begin{equation*}
    v^{\tilde{\cal{I}}_{n - 1} \setminus i, j}_{\tilde{\varrho}_n}\Bigl(\bf{X}^{\tilde{\cal{I}}_{n - 1} \setminus i}_{\tilde{\varrho}_n}\Bigr) = v^{\tilde{\cal{I}}_{n - 1}, j}_{\tilde{\varrho}_n}\bigl(\bf{X}^{\tilde{\cal{I}}_{n - 1}}_{\tilde{\varrho}_n}\bigr) = v^{\cal{A}, j}_{\tilde{\varrho}_n}\bigl(\bf{X}^{\cal{A}}_{\tilde{\varrho}_n}\bigr)
\end{equation*}
for $j \in \tilde{\cal{I}}_{n - 1}$. From this and \eqref{eq:dead_at_step_n_equality}, we derive $X^i_{\tilde{\varrho}_n} = v^{\cal{A}, j}_{\tilde{\varrho}_n}\bigl(\bf{X}^{\cal{A}}_{\tilde{\varrho}_n}\bigr)$ in contradiction with \eqref{eq:plus_not_occured}. In conclusion, it must hold that $\tau + (\tau_+ \land \varsigma) \leq \tilde{\varrho}_n$. Thus, if we can show that $\pr(\tau_- < \tau_+,\, \tau < \tilde{\varrho}_n) > 0$, it follows that $\inf_{\varrho_n \leq t \leq \tilde{\varrho}_n} X^k_t \leq -1 < 0$ for some $k \in \cal{D} \subset \tilde{\cal{I}}_{n - 1}$ with positive probability, as desired.

Now, conditional on $\F_{\tau}$, the process $\tilde{\bf{X}}^{\cal{D}}$ is independent of $\bf{B}^{\cal{D}}$ and $\varsigma$. Thus, we can apply Lemma \ref{lem:hitting_time} with the choice $f^i_-(t) = -\sigma^{-1}$ and $f^i_+(t) = \sigma^{-1} \bf{B}^i_t$, $i \in \cal{D}$, and $S = \varsigma$. This implies that $\pr$-a.s.\@ it holds that $\pr(\tau_- < \tau_+ \vert \F_{\tau}) > 0$ if $\tau < T$. Next, as in \eqref{eq:no_killing_at_terminal}, we have that $\pr(\tilde{\varrho}_n = T,\, \varrho < T) = 0$. Consequently, up to a $\pr$-nullset, we have that $\tau < \tilde{\varrho}_n$ implies that $\tau < T$. From this and the above application of Lemma \ref{lem:hitting_time} it follows that
\begin{equation*}
    \pr\bigl(\tau_- < \tau_+,\, \tau < \tilde{\varrho}_n \bigr) = \ev\bigl[\pr(\tau_- < \tau_+ \vert \F_{\tau}) \bf{1}_{\{\tau < \tilde{\varrho_n}\}}\bigr] > 0,
\end{equation*}
where we used in the last equality that $\pr(\tau < \tilde{\varrho_n}) > 0$ by the choice of $\epsilon$ in the definition of $\tau$. This concludes the proof of the first part of the claim, namely that $\tilde{\varrho}_n = \varrho_n$. It remains to show that $\tilde{\cal{I}}_n = \cal{I}^{n - 1}_n$.

\textit{Proof of} $\tilde{\cal{I}}_n = \cal{I}^{n - 1}_n$\textit{:} If $\tilde{\varrho}_n < \infty$, then a particle is removed at time $\tilde{\varrho}_n = \varrho_n$, both from the system $(X^i, \tilde{Y}^i, \tilde{Z}^i)_{i \in [N]}$ and the system $(X^i, Y^{(n - 1), i}, Z^{(n - 1), i})_{i \in [N]}$. Let us denote the corresponding indices by $i_1$ and $i_2 \in \tilde{\cal{I}}_{n - 1}$, respectively. Since $\tilde{\cal{I}}_n = \tilde{\cal{I}}_{n - 1} \setminus i_1$ and $\cal{I}^{n - 1}_n = \tilde{\cal{I}}_{n - 1} \setminus i_2$, it suffices to show that $i_1 = i_2$. At time $\tilde{\varrho}_n$, since $i_1 \in \cal{I}^{\tilde{\cal{I}}_{n - 1}}\bigl(\tilde{\varrho}_n, \bf{X}^{\tilde{\cal{I}}_{n - 1}}_{\tilde{\varrho}_n}\bigr)$, we have by the boundary condition of PDE \eqref{eq:pde_moving_boundary} that
\begin{align*}
    \tilde{Y}^j_{\tilde{\varrho}_n} = Y^{(n), j}_{\tilde{\varrho}_n} = v^{\tilde{\cal{I}}_{n - 1} \setminus i_1, j}_{\tilde{\varrho}_n}\Bigl(\bf{X}^{\tilde{\cal{I}}_{n - 1} \setminus i_1}_{\tilde{\varrho}_n}\Bigr) = v^{\tilde{\cal{I}}_{n - 1}, j}_{\tilde{\varrho}_n}\bigl(\bf{X}^{\tilde{\cal{I}}_{n - 1}}_{\tilde{\varrho}_n}\bigr).
\end{align*}
An analogous argument implies $Y^{(n - 1), j}_{\tilde{\varrho}_n} = v^{\tilde{\cal{I}}_{n - 1}, j}_{\tilde{\varrho}_n}\bigl(\bf{X}^{\tilde{\cal{I}}_{n - 1}}_{\tilde{\varrho}_n}\bigr)$. From this we deduce that
\begin{equation*}
    X^{i_1}_{\tilde{\varrho}_n} \leq \sum_{j = 1}^N D_{i_1j} \tilde{Y}^j_{\tilde{\varrho}_n} = \sum_{j = 1}^N D_{i_1j} Y^{(n - 1), j}_{\tilde{\varrho}_n},
\end{equation*}
so particle $i_1$ is also dead in the system $(X^i, Y^{(n - 1), i}, Z^{(n - 1), i})_{i \in [N]}$. Symmetrically, $i_2$ is dead in the system $(X^i, \tilde{Y}^i, \tilde{Z}^i)_{i \in [N]}$. However, both $i_1$ and $i_2$ are the minimal elements of $\tilde{\cal{I}}_{n - 1}$ with that property, so we conclude that $i_1 = i_2$. This finishes the proof.
\end{proof}

\section{The Mean-Field Limit} \label{sec:mfl}

We conclude the paper by stating some preliminary observations regarding the mean-field limit of FBSDE \eqref{eq:fbsde_hetero} in the symmetric setting $D_{ij} = \frac{\alpha}{N}$ as the number of particles tends to infinity. First, note that the symmetry allows one to write the system in a simplified form. Indeed, if we set
\begin{equation*}
    \bar{Y}^N = \frac{1}{N} \sum_{j = 1}^N Y^i \quad \text{and} \quad \bar{Z}^N = \frac{1}{N} \sum_{j = 1}^N Z^i,
\end{equation*}
then FBSDE \eqref{eq:fbsde_hetero} becomes
\begin{equation} \label{eq:fbsde_avg}
    \d X^i_t = \sigma \, \d W^i_t, \quad 
    \d \bar{Y}^N_t = \bar{Z}^N_t \cdot \d \bf{W}_t
\end{equation}
with $X^i_t = \xi_i$ and $\bar{Y}^N_T = \frac{1}{N} \sum_{j = 1}^N\bf{1}_{\tau_j \leq T}$, where
\begin{equation*}
    \tau_i = \inf\bigl\{t \in [0, T] \define X^i_t \leq \alpha \bar{Y}^N_t\bigr\}.
\end{equation*}
In this way, the $N$ equations for $Y^1$,~\ldots, $Y^N$ in the backward part of FBSDE \eqref{eq:fbsde_hetero} get replaced by a single equation for $\bar{Y}^N$. 

If we suppose that the killing times $\tau_i$ of the particles asymptotically decorrelate, in the sense that
\begin{equation*}
    \pr(\tau_i \leq T,\, \tau_j \leq T) - \pr(\tau_i \leq T) \pr(\tau_j \leq T) \to 0
\end{equation*}
as $N \to \infty$ for $i \neq j$, then the martingale $\int_0^{\cdot} \bar{Z}^N_t \cdot \d \bf{W}_t$ vanishes in the limit. As a result, the limit $Y_t$ of $\bar{Y}^N_t$ is simply given by $Y_t = \pr(\tau \leq T)$ for $t \in [0, T]$, where $\tau$ is the killing time of the representative particle in the mean-field limit. Passing to the limit in the definition of $\tau_i$, suggests that this killing time is given by
\begin{equation*}
    \tau = \inf\bigl\{t \in [0, T] \define X_t \leq \alpha Y_t\bigr\} = \inf\Bigl\{t \in [0, T] \define X_t \leq \alpha \pr(\tau \leq T)\Bigr\},
\end{equation*}
where the state $X_t$ of the representative particle in the mean-field limit follows the dynamics
\begin{equation*}
    \d X_t = \sigma \, \d W_t
\end{equation*}
with $X_0 = \xi$ for a Brownian motion $W$. Hence, the mean-field limit is completely determined by the probability $\pr(\tau \leq T)$, which satisfies the fixed-point equation
\begin{equation} \label{eq:mfl_fpe}
    \pr(\tau \leq T) = \pr\Bigl(\min_{t \in [0, T]} X_t \leq \alpha \pr(\tau \leq T)\Bigr).
\end{equation}
As we shall demonstrate below, this equation may have more than one solution, in contrast to the finite particle for which uniqueness attains.

First, however, let us note that at this stage we have no theoretical support for the supposition that the killing times $\tau_i$ become asymptotically independent. In the absence of asymptotic independence, the limiting behaviour of FBSDE \eqref{eq:fbsde_avg} is much harder to capture. First of all, it is not obvious whether and in what sense the process $\bar{Y}^N$ converges nor if its limit should be continuous. If the limit $Y$ were discontinuous, stating the FBSDE satisfied by $X$ and $Y$ becomes a subtle enterprise, since fluctuations of the process $\bar{Y}^N$ present in the pre-limit system get lost in the jumps of $Y$. Consequently, setting $\tau = \inf\{t \in [0, T] \define X_t \leq \alpha Y_t\}$ would lead to an underestimation of the proportion of absorptions in the mean-field system, i.e.\@
\begin{equation*}
    \pr(\tau \leq T) < \ev[Y_T].
\end{equation*}
Instead, one should ``decorate'' the jumps of $Y$ with the fluctuations stemming from the pre-limit system. That is, each jump time $t \in [0, T]$ comes attached with a random interval $[Y^-_t, Y^+_t]$ including $Y_{t-}$ and $Y_t$ and containing the asymptotic fluctuations. Then the representative particle is killed between $t-$ and $t$ if $X_t \leq \alpha Y^-_t$. Setting $Y^-_t = Y^+_t = Y_t$ for continuity points $t \in [0, T]$ of $Y$, the correct definition of the killing time $\tau$ would then by
\begin{equation*}
    \tau = \inf\bigl\{t \in [0, T] \define X_t \leq \alpha Y^-_t\bigr\}.
\end{equation*}
The fixed-point condition \eqref{eq:mfl_fpe} now becomes
\begin{equation}
    Y_t = \pr(\tau \leq T \vert \F^Y_t),
\end{equation}
where $\bb{F}^Y = (\F^Y_t)_{t \in [0, T]}$ is given by $\F^Y_t = \sigma(Y_s, Y^-_s, Y^+_s \define s \in [0, t])$. This would appear to be a rather complicated system, an analysis of which we leave for future research.

Let us come back to the simpler fixed-point equation \eqref{eq:mfl_fpe} characterising the mean-field limit under asymptotic independence of the killing times. We shall show that in many cases, it admits several solutions. For a given $p \in [0, 1]$, we compute
\begin{align} \label{eq:fp_calcuation}
    \pr\Bigl(\min_{t \in [0, T]} \leq \alpha p\Bigr) &= \pr\biggl(\min_{t \in [0, T]} W_t \leq \frac{\alpha p - \xi}{\sigma}\biggr) \notag \\
    &= \int_{[0, \infty)} \pr\biggl(\min_{t \in [0, T]} W_t \leq \frac{\alpha p - x}{\sigma}\biggr) \, \d \L(\xi)(x) \notag \\
    &= \int_{[0, \infty)} 2\pr\biggl(W_T \leq \frac{(\alpha p - x) \land 0}{\sigma}\biggr) \, \d \L(\xi)(x) \notag \\
    &= \int_{[0, \infty)} 2\pr\biggl(W_1 \leq \frac{(\alpha p - x) \land 0}{\sigma\sqrt{T}}\biggr) \, \d \L(\xi)(x),
\end{align}
where we used the reflection principle in the third equality. Now, let us consider the case $T = 1$, $\sigma = 1$, and $\xi = \alpha$. Regardless of the choice of $\alpha$, $p_+ = 1$ always yields a fixed point of the map
\begin{equation*}
    p \mapsto \pr\Bigl(\min_{t \in [0, T]} \leq \alpha p\Bigr),
\end{equation*}
corresponding to the solution of the mean-field limit with absorption time $\tau_+ = 0$. Next, note that for the specified parameters, the expression on the right-hand side of \eqref{eq:fp_calcuation} becomes $2\pr(W_1 \leq -(1 - p)\alpha)$. Hence, setting $p = \frac{1}{2}$, we can select $\alpha > 0$ such that $\pr(W_1 \leq -\frac{\alpha}{2}) = \frac{1}{4}$. With this choice of $\alpha$, $p_- = \frac{1}{2}$ is another fixed point with associated killing time $\tau_- = \inf\{t > 0 \define X_t \leq \frac{\alpha}{2}\}$.

Note that for the chosen parameters, all particles in the finite system are immediately killed, so that the particle system trivially converges to the maximal solution $p_+ = 1$ of the fixed-point equation \eqref{eq:mfl_fpe}.

\appendix

\printbibliography

\end{document}